\documentclass{amsart}

\usepackage{verbatim}

\usepackage{amsthm}
\usepackage{amsfonts}
\usepackage{amsmath}
\usepackage{amssymb}
\usepackage{verbatim, color}

\usepackage{pictex}
\usepackage{url}

\usepackage{enumerate} 

\usepackage{amsmath,amssymb,url}

\numberwithin{equation}{section}

\theoremstyle{plain}
\newtheorem{theorem}{Theorem}[section]
\newtheorem{lemma}[theorem]{Lemma}

\newtheorem{corollary}[theorem]{Corollary}
              
\theoremstyle{definition}
\newtheorem{definition}[theorem]{Definition}

\newtheorem{remark}[theorem]{Remark}

  \newcommand{\C}[1]{\mathcal{#1}}

\begin{document}
\title[Implicator groupoids]{On implicator groupoids}

\author[J. M. Cornejo]{Juan M. Cornejo }
\email{jmcornejo@uns.edu.ar}
\address{Departamento de Matem\'{a}tica\\
Universidad Nacional del Sur\\ INMABB \\ CONICET \\ Bahia Blanca\\Argentina}

\author[H. P. Sankappanavar]{Hanamantagouda P. Sankappanavar}
\email{sankapph@newpaltz.edu}
\address{Department of Mathematics\\
 State University of New York\\New Paltz\\NY 12561\\USA}
  
\thanks{The first author wants to thank the
institutional support of CONICET  (Consejo Nacional de Investigaciones 
Cient\'ificas y T\'ecnicas).}



\begin{abstract}
In a paper published in 2012, the second author 
extended the well-known fact that Boolean algebras can be defined using only implication and a constant, to De Morgan 
algebras---this 
result led him to introduce, and investigate (in the same paper), the variety $\C{I}$ of algebras, there called implication zroupoids ({\bf{I}}-zroupoids) and here called implicator gruopids ($\C{I}$-groupoids), 
that generalize De Morgan algebras. 

	The present paper is a continuation of the 
paper mentioned above  	
	and is devoted to investigating  the structure of the lattice of subvarieties of $\C{I}$,
and also to making further contributions to the theory of implicator groupoids.  Several new subvarieties of $\C{I}$ are introduced and their relationship with each other, and with the subvarieties of $\C{I}$ which were already investigated in the paper mentioned above, are explored.
\end{abstract}
\maketitle

\allowdisplaybreaks 

\section{Introduction} 

Boolean algebras can be defined using only implication and a constant.   In 2012, this result was extended to De Morgan algebras in \cite{sankappanavarMorgan2012} which led the second author of this paper to introduce, and investigate, the variety  
$\C{I}$ of \emph{implicator groupoids} (there called implication zroupoids) that generalize De Morgan algebras.  

\begin{definition}\label{D:IG}
An algebra $\mathbf A = \langle A, \to, 0 \rangle$, where $\to$ is binary and $0$ is a constant,   
and $x'$ is defined by $x' : = x \to 0$, is called an 
 \emph{implicator groupoid} ($\C{I}$-groupoid, for short) if $\mathbf A$ satisfies:
\begin{enumerate} 
	\item[(I)] 	$(x \to y) \to z \approx [(z' \to x) \to (y \to z)']'$, 
	\item[(I$_{0}$)] \quad $ 0'' \approx 0$.
\end{enumerate}
\end{definition}

(The term ``implicator'' has been extensively used in Fuzzy logic and Fuzzy set theory; see for example \cite{DeVeCoGo13}.)

Throughout this paper $\C{I}$ denotes the variety of  \emph{implicator groupoids}.

The present paper is a continuation of \cite{sankappanavarMorgan2012} and is devoted to investigating the structure of the lattice of subvarieties of $\C{I}$, and also to making further contributions to the theory of implication groupoids. 
Several new subvarieties of $\C{I}$ are introduced and their relationship with each other and with the subvarieties of $\C{I}$, which were already investigated in \cite{sankappanavarMorgan2012}, are explored.
The paper is organized as follows.

\begin{definition}
The varieties $\C{MID}$, $\C{JID}$, and $\C{I}_{2,0}$ are defined relative to $\C{I}$, respectively, by:  
\begin{enumerate}
	\item[(MID)] \quad $x \land x \approx x$,
	\item[(JID)] \quad $x \lor x \approx x$,
	\item[(I$_{2,0}$)] \quad $x'' \approx x$.
\end{enumerate}
\end{definition}

In Section \ref{S:3}, we focus on the variety $\C{I}_{2,0}$
We present several fundamental identities that hold in $\C{I}_{2,0}$ and that play an essential role in the rest of the paper.  We prove that $\C{I}_{2,0} = \C{MID} = \C{JID}$.  We  also show that the   implicator groupoids which are implication algebras in the sense of Abbott are precisely Boolean algebras.

 In this paper we use the characterizations of De Morgan algebras, Kleene algebras and Boolean algebras obtained in \cite{sankappanavarMorgan2012} as definitions.

\begin{definition} \label{definit_DM_KL_BA}
 $\mathbf A \in \C{I}$ is a \emph{De Morgan algebra} ($\mathbf{DM}$-algebra for short) if $\mathbf A$ satisfies the axiom:
\begin{enumerate}  
	\item[(DM)] \quad $(x \to y) \to x \approx x$.
\end{enumerate}
 A $\mathbf{DM}$-algebra $\mathbf{A}$ is a {\it Kleene algebra} ($\mathbf{KL}$-algebra for short) if
$\mathbf A$ satisfies either of the equivalent axioms:
\begin{enumerate} 
	\item[(KL$_1$)] \quad $(x \to x) \to (y \to y)' \approx x \to x$,
	\item[(KL$_2$)] \quad $(y \to y) \to (x \to x) \approx x \to x$.
\end{enumerate}
 A $\mathbf{DM}$-algebra $\mathbf{A}$ is a {\it Boolean algebra} ($\mathbf{BA}$-algebra for short) if
$\mathbf A$ satisfies the axiom:
\begin{enumerate} 
	\item[(BA)] \quad $x \to x \approx 0'$.
\end{enumerate}
We denote by $\C{DM}$, $\C{KL}$, and $\C{BA}$, respectively, the varieties of $\mathbf{DM}$-algebras,
$\mathbf{KL}$-algebras, and $\mathbf{BA}$-algebras. 
\end{definition}

In Section \ref{S:4}, Boolean algebras and Kleene algebras are characterized as suitable subvarieties of $\C{I}_{2,0}$. 
In Section \ref{S:5}, we prove a Glivenko-like theorem for algebras in $\C{I}$.

Throughout this paper we use the following definitions:
\[
	\textup{(M)} \quad 	x \land y := (x \to y')' \quad\text{ and }\quad 
	\textup{(J)} \quad  x \lor y := (x' \land y')'.
\]
With each  $\mathbf A \in \C{I}$, we associate the following algebra:
\[
	\mathbf{A{^{mj}}} := \langle A, \land, \lor, 0 \rangle.  
\]

\begin{definition}
The varieties $\C{SCP}$ and $\C{MC}$ are defined relative to $\C{I}$, respectively, by:
\begin{enumerate}
	\item[(SCP)] \quad $x \to y \approx y' \to x'$,
	\item[(MC)] \quad $x \land y \approx y \land x$.
\end{enumerate}
\end{definition}

Section \ref{S:6} discusses the relationship between the varieties $\C{SCP}$ and $\C{MC}$.  It is shown that $\C{SCP} \subset \C{MC}$, and, moreover, if we restrict them to $\C{I}_{2,0}$, they coincide with each other. 
 
 Section \ref{S:7} determines the  algebras 
$\mathbf{A} \in \C{I}$ such that the induced algebra $\mathbf{A^{mj}} = \langle A, \land, \lor, 0 \rangle$ is a lattice.  Surprisingly, it turns out that this is the case precisely when the absorption identity holds, which, in turn, is equivalent to $\mathbf{A}$ being a De Morgan algebra.  The section ends with another characterization of De Morgan algebras as a subvariety of $\C{I}_{2,0}$.  In Section \ref{S:8}, we give several interesting properties of $\C{I}$ and of the variety $\C{I}_{2,0}\, \cap\,  \C{MC}$. 

\begin{definition}
The varieties $\C{Z}$, $\C{C}$,  $\C{A}$, $\C{I}_{3,1}$, $\C{I}_{1,0}$, and $\C{ID}$ are defined relative to $\C{I}$, respectively, by:
\begin{enumerate}
	\item[(Z)] \quad $x \to y \approx 0$,
	\item[(C)] \quad $ x  \to  y \approx y \to x$,
	\item[(CP)] \quad $ x  \to  y' \approx y \to x'$,
	\item[(A)] \quad $(x \to y) \to z \approx x \to (y \to z)$,
	\item[({I}$_{3,1}$)] \quad $x''' \approx x'$,
	\item[({I}$_{1,0}$)] \quad $x' \approx x$,
	\item[(ID)] \quad $x \to x \approx x$.
\end{enumerate}
\end{definition}

Section \ref{S:9} investigates relationships among the varieties $\C{Z}$, $\C{C}$,  $\C{A}$, and $\C{I}_{3,1}$.  It is proved that  $\C{Z} \subset \C{C} \subset \C{A} \subset \C{I}_{3,1}$. In Section \ref{S:10}, it is shown that $\C{I}_{1,0} = \C{ID}\, \cap\, \C{A}$.

In Section \ref{S:11}, the three $2$-element algebras $\mathbf{2}_z$, $\mathbf{2}_s$, and $\mathbf{2}_b$ of $\C{I}$ are recalled and the varieties they generate, denoted $\C{V}(\mathbf{2}_z)$, $\C{V}(\mathbf{2}_s)$,and $\C{V}(\mathbf{2}_b)$, respectively, are characterized.  As a consequence, the variety of join semilattices with least element, denoted $\C{SL}$, can be viewed as $\C{C} \cap \C{I}_{1,0} \subset \C{I}$, from which it follows, in view of a well-known result (see \cite{FrNa73}) that the congruence lattices of  algebras in $\C{I}$ do not satisfy any nontrivial lattice identities.

Section \ref{S:12} proves that  $\C{MC} \cap \C{MID}  \cap  \C{A} =  \C{SL}    \subset  \C{CP}$. In Section \ref{S:13}, we prove that $\C{MC}  \cap  \C{ID}  =  \C{C}  \cap   \C{I}_{1,0}= \C{V}(2_s)$.  The paper concludes with some remarks about results that are contained in
\cite{CoSa2015b}, \cite{CoSa2015c}, and 
\cite{CoSa2015a}.

\section{The variety $\C{I}_{2,0}$}\label{S:3}

In this section, we focus on the variety $\C{I}_{2,0}$, an important subvariety of $\C{I}$.  We present several fundamental identities that hold in $\C{I}_{2,0}$ and play an essential role in the rest of this paper. 
We also give two new characterizations of $\C{I}_{2,0}$ and show that the  algebras in $\C{I}$ that are implication algebras in the sense of Abbott are precisely Boolean algebras.

The following lemmas are frequently used in the sequel.
\begin{lemma}
[{\rm \cite[{\rm Theorem 8.15}]{sankappanavarMorgan2012}}]
\label{general_properties_equiv}
	Let  $\mathbf A \in \C{I}$. Then the following are equivalent in $\mathbf A$:
	\begin{enumerate} [{\rm (a)}]
		\item   $0' \to x \approx x$,  \label{TXX} 
		\item   $x'' \approx x$, 
		\item   $(x \to x')' \approx x$, \label{reflexivity} 
		\item   $x' \to x \approx x$. \label{LeftImplicationwithtilde}
	\end{enumerate}
\end{lemma}

\begin{lemma} \cite[Lemma 8.13 (e),(f)] {sankappanavarMorgan2012} \label{general_properties}
	Let $\mathbf A \in \C{I}_{2,0}$. Then $\mathbf A$ satisfies:
	\begin{enumerate}[{\rm(a)}]
		\item  $x' \to 0' \approx 0 \to x$, \label{cuasiConmutativeOfImplic2}
		\item  $0 \to x' \approx x \to 0'$. \label{cuasiConmutativeOfImplic}
	\end{enumerate}
\end{lemma}

\begin{lemma} \cite[Theorem 8.15]{sankappanavarMorgan2012} \label{SankaTheorem8.15}
Let  $\mathbf{A} \in \C{I}_{2,0}$.  Then $\mathbf{A} \models x \land x \approx x$.
\end{lemma}

The following lemma is a restatement of \cite[Theorem 5.1]{sankappanavarMorgan2012}.

\begin{lemma} 
\label{SankaTheorem5.1}
Let  $\mathbf{A} \in \C{I}$.  Then 
\begin{center}
$\mathbf{A}$ is a De Morgan algebra iff  
$\mathbf{A} \models (0 \to x) \to y \approx y$.
\end{center}
\end{lemma}

\begin{lemma} \cite[Lemma 7.5b]{sankappanavarMorgan2012} \label{SankaLemma7.5}
Let  $\mathbf{A} \in \C{I}_{2,0}$.  Then $\mathbf{A} \models (x \to y'')' \approx (x \to y)'$.
\end{lemma}

\begin{lemma} \cite[Theorem 7.6]{sankappanavarMorgan2012} \label{SankaTheo7.6}
	Let  $\mathbf{A} \in \C{I}$.  Then $\mathbf{A} \models x''' \to y \approx x' \to y$.
\end{lemma}

The following lemma plays a crucial role in the rest of the paper.

\begin{lemma} \label{general_properties2}
	Let $\mathbf A \in \C{I}_{2,0}$. Then $\mathbf A$ satisfies:
	\begin{enumerate}[{\rm(1)}]  
		\item $(x \to 0') \to y \approx (x \to y') \to y$, \label{281014_05}
		\item $x \to (0 \to x)' \approx x'$, \label{291014_02}
		\item $(y \to x) \to y \approx (0 \to x) \to y$, \label{291014_10}
		\item $[(x \to 0') \to y]' \approx (0 \to x) \to y'$, \label{071114_02}
		\item $0 \to x \approx 0 \to (0 \to x)$, \label{311014_03}
		\item $[x' \to (0 \to y)]' \approx (0 \to x) \to (0 \to y)'$, \label{031114_06}
		\item $[x \to (y \to x)']' \approx (x \to y) \to x$, \label{291014_09}
		\item $0 \to [(0 \to x) \to (0 \to y')'] \approx 0 \to (x \to y)$, \label{031114_05}
		\item $0 \to \{(0 \to x) \to y'\} \approx x \to (0 \to y')$, \label{071114_03}
		\item $0 \to (0 \to x)' \approx 0 \to x'$, \label{031114_07}
		\item $0 \to (x \to y) \approx x \to (0 \to y)$, \label{071114_04}
		\item $[(0 \to x) \to y] \to x \approx [(y \to x) \to (0 \to x)']'$, \label{291014_03}
		\item $[\{(0 \to x) \to y\} \to x]' \approx (y \to x) \to (0 \to x)'$, \label{291014_05}
		\item $(0 \to x') \to (y \to x) \approx y \to x$, \label{281014_07}
		\item $x' \to (0 \to x)  \approx 0 \to x$, \label{291014_04}
		\item $(y \to x)' \approx (0 \to x) \to (y \to x)'$, \label{291014_06}
		\item $(x \to y) \to (0 \to y)' \approx (x \to y)'$, \label{291014_07}
		\item $0 \to (x' \to y)' \approx x \to (0 \to y')$, \label{071114_01}
		\item $[(x \to y) \to x] \to [(y \to x) \to y] \approx x \to y$, \label{271114_03}
		\item $[x \to (y \to x')] \to x \approx x' \to (y \to x')'$, \label{271114_04}
		\item $x \to (y \to x') \approx y \to x'$, \label{281114_01}
		\item $0 \to (x \to y')' \approx 0 \to (x' \to y)$, \label{191114_05}
		\item $(x \to y) \to y' \approx y \to (x \to y)'$, \label{071114_05}		
		\item $x \to [(y \to z') \to x]' \approx (x' \to y) \to \{(0 \to z) \to x'\}$, \label{181114_04}
		\item $[\{((x \to 0') \to y) \to z\} \to \{u \to ((0 \to x) \to y')\}']' \\
		\approx (z \to u) \to \{(0 \to x) \to y'\}$, \label{181114_05}
		\item $(z \to x) \to (y \to z) \approx (0 \to x) \to (y \to z)$, \label{080415_01}
		\item $(x' \to y) \to [(0 \to z) \to x'] \approx (0 \to y) \to [(0 \to z) \to x']$, \label{181114_07}
		\item $x \to [(y \to z') \to x]' \approx (0 \to y) \to [(0 \to z) \to x']$, \label{181114_10}		
		\item $(x' \to y) \to (x \to y') \approx x \to y'$, \label{080415_02}
		\item $(x \to y')' \to (x' \to y)' \approx x \to y'$. \label{080415_03}
	\end{enumerate}
\end{lemma}

\begin{proof}
Please see the appendix for the proofs of all these items.				
\end{proof}

\begin{lemma} \label{Lemma_300315_01}
	Let  $\mathbf A \in \C{I}$.  Then $\mathbf A$ satisfies:
	\begin{enumerate} [{ \rm(1)}]
		\item  $[(x \to y) \to z]''' \approx [(x \to y) \to z]'$, \label{300315_01}
		\item  $(x \to y) \to z \approx  [(x \to y) \to z]''$, \label{060415_01}
		\item  $(x \to y)' \approx (x'' \to y)'$, \label{300315_02}
		\item  $x \land y \approx (x \land y)''$, \label{310315_03}
		\item  $x \lor y \approx (x \lor y)''$, \label{310315_05}
		\item  $x \land y \approx (x' \lor y')'$. \label{310315_04}
	\end{enumerate}
\end{lemma}

\begin{proof} 
(1): \quad $[(x \to y) \to z]''' \overset{(I)}{\approx} [(z' \to x) \to (y \to z)']''''$\\

\noindent$\overset{ \ref{SankaTheo7.6}}{\approx}  [(z' \to x) \to (y \to z)']'' \overset{(I)}{\approx} [(x \to y) \to z]'$.\\
		
(2): \quad $(x \to y) \to z \overset{(I)}{\approx} [(z' \to x) \to (y \to z)']'$\\

\noindent$\overset{(\ref{300315_01})}{\approx} [(z' \to x) \to (y \to z)']''' \overset{(I)}{\approx} [(x \to y) \to z]''$.		
		
(3):	\quad $(x \to y)'	 \approx  (x \to y) \to 0  \overset{(I)}{\approx}  [(0' \to x) \to (y \to 0)']'$
\begin{align*}
&\approx  [\{(0 \to 0) \to x\} \to (y \to 0)']' \overset{(I)}{\approx}  [\{(x' \to 0) \to (0 \to x)'\}' \to (y \to 0)']' \\
&\overset{\ref{SankaTheo7.6}}{\approx}  [\{(x''' \to 0) \to (0 \to x)'\}' \to (y \to 0)']'  \\
&\overset{\ref{SankaLemma7.5}}{\approx}  [\{(x''' \to 0) \to (0 \to x'')'\}' \to (y \to 0)']' 
 \overset{(I)}{\approx}  [\{(0 \to 0) \to x''\} \to (y \to 0)']'  \\
&\approx  [(0' \to x'') \to (y \to 0)']'  \overset{(I)}{\approx}  (x'' \to y) \to 0  \approx  (x'' \to y)'. 
\end{align*} 

(4): \quad $x \lor y \approx  (x' \land y')'  \approx  (x' \to y'')''   \approx  [(x \to 0) \to y'']''$\\ 

$\overset{(\ref{300315_01})}{\approx}  [(x \to 0) \to y'']''''   \approx  (x' \to y'')'''' \approx  (x' \land y')'''  
 \approx  (x \lor y)''$. \\

(5): \quad $x \land y	\overset{(M)}{\approx}  (x \to y')'  \overset{(\ref{300315_02})}{\approx}  (x'' \to y')' \approx  [(x' \to 0) \to y']'$\\

$\overset{(\ref{300315_01})}{\approx}  [(x' \to 0) \to y']''' \approx  (x'' \to y')''' \overset{(\ref{300315_02})}{\approx}  (x \to y')''' 
		 \overset{(M)}{\approx}  (x \land y)''$.\\
		
(6): \quad $x \land y	\overset{(M)}{\approx}  (x \to y')' 
\overset{(\ref{300315_02})}{\approx}  (x'' \to y')' \approx  [(x' \to 0) \to y']'$\\

$\overset{(\ref{300315_01})}{\approx}  [(x' \to 0) \to y']'''  \approx  (x'' \to y')''' 
\overset{\ref{SankaLemma7.5}}{\approx}  (x'' \to y''')''' $\\
$\overset{(M)}{\approx}  (x'' \land y'')''  \approx  (x' \lor y')'$. 
\end{proof}

The following theorem gives two new characterizations of $\C{I}_{2, 0}$.
\begin{theorem} \label{T3.5}
	Let $\mathbf{A} \in \C{I}$    
	Then
		 $ \C{I}_{2,0} = \C{MID} =  \C{JID} $.
\end{theorem}

\begin{proof}
From Lemma \ref{SankaTheorem8.15}, we get $ \C{I}_{2,0} \subseteq \C{MID}$. Using $x \land x \approx x$, we get\\
\quad  $x \lor x	 \approx  (x' \land x')'  
   	\overset{hyp}{\approx} x'' 
	\overset{hyp}{\approx}  (x \land x)''   
	\approx  (x \to x')'''  $
\begin{align*} 
	&\overset{\ref{Lemma_300315_01} (\ref{300315_02})} {\approx}  (x'' \to x')'''   
	\approx  [(x' \to 0) \to x']'''    	
\overset{\ref{Lemma_300315_01} (\ref{300315_01})}{\approx}   [(x' \to 0) \to x']' \\
&\approx (x'' \to x')' \overset{\ref{Lemma_300315_01} (\ref{300315_02})}{\approx}  (x \to x')' 
\overset{}{\approx} x \land x \overset{(MID)}{\approx} x,  
\end{align*}  	
 proving $\C{MID} \subseteq \C{JID}$.   Finally, from $x \lor x \approx x$ and 	
 Lemma \ref{SankaTheo7.6}, we have
 $x'' \approx (x \lor x)'' \approx (x' \land x')''' \approx (x' \to x'')'''' \approx (x' \to x'')'' \approx (x' \land x')' \approx x \lor x \approx x$, thus $\C{JID} \subseteq \C{I}_{2, 0}$.
\end{proof}

In the sequel we will use these names interchangeably.

We conclude this section by giving a characterization of those algebras of $\C{I}$ that are implication algebras in the sense of Abbott \cite{abbottSemi1967}.

  An algebra $\mathbf{A} = \langle A, \to 0 \rangle$ is an {\it implication algebra in the sense of 
Abbott}
\cite{abbottSemi1967} if it satisfies:
\begin{enumerate} 
	\item[(1)] $(x \to y) \to x \approx x$, 
	\item[(2)] $(x \to y) \to y \approx (y \to x) \to x$,  
	\item[(3)] $x \to (y \to z) \approx y \to (x \to z)$. 
\end{enumerate}

\begin{theorem}
	$\mathbf{A} \in \C I$ is an implication algebra in the sense of 
Abbott	
	iff $A$ is a Boolean algebra.
\end{theorem}

\begin{proof}	
	Let $\mathbf{A} \in \C{I}$. If $\mathbf{A}$ is a Boolean algebra, then clearly $\mathbf{A}$ is an Implication algebra in the sense of Abbott.  For the converse, let $\C{V}$ be the subvariety of $\C{I}$ defined by 
(1), (2) and~(3). 
Since $\C{V}$ satisfies (1),  
clearly  $\C{V} \subseteq
\C{DM}$.   To finish off the proof, it is sufficient to observe that the $3$-element Kleene algebra and the $4$-element De Morgan algebra fail to satisfy (2).  
\end{proof}

\section{The varieties $\C{BA}$ and $\C{KL}$ as subvarieties of  $\C{MID}$}\label{S:4}

In this section we describe the varieties of Boolean algebras and Kleene algebras as subvarieties of the variety $\C{MID}$.
In the next theorem, Boolean algebras are 
characterized
relative to $\C{MID}$.

\begin{theorem}
	Let $\mathbf A \in \C{MID}$.  Then the following are equivalent in $\mathbf{A}$:
	\begin{enumerate} [{ \rm (a)}]
		\item   \label{010715_04}
			$\mathbf A \models x \land  x' \approx  0$, 
		\item  \label{010715_06} 
		 $\mathbf{A}$ is a Boolean algebra.
	\end{enumerate}
\end{theorem}

\begin{proof} 
		 (\ref{010715_04}) $\Rightarrow$ (\ref{010715_06}):  $0' \approx (x \wedge x')' = (x \to x'')'' \approx x \to x$ and, hence,
 $(0 \to x) \to y  \overset{\ref{general_properties} (\ref{cuasiConmutativeOfImplic})}{\approx}  (x' \to 0') \to y     
		\approx  \{x' \to (x \to x)\} \to y    
\noindent	\overset{hyp}{\approx}  \{x' \to (x \to x'')\} \to y    
		\overset{\ref{general_properties2} (\ref{281114_01})}{\approx}  (x \to x) \to y  
		\approx  0' \to y   
		\overset{\ref{general_properties_equiv} (\ref{TXX})}{\approx}  y. $  		
  So, it follows from  
Lemma \ref{SankaTheorem5.1} and (a) that (b) is true,		
while (\ref{010715_06}) $\Rightarrow$ (\ref{010715_04}) is straightforward. 
\end{proof}

Next, we present an axiomatization of Kleene algebras as a subvariety of $\C{I}_{2,0}$.  For this purpose, we need the following result.

\begin{lemma}\label{L:aux}
	Let $\mathbf{A} \in \C{I}_{2,0}$ and $\mathbf{A}$ satisfy $(x \to x) \to (y \to y) \approx y \to y$.  Then
	$\mathbf{A} \models [x \to (x' \to x')] \to 0' \approx 0 \to x$.
\end{lemma}

\begin{proof}
\quad $[x \to (x' \to x')] \to 0'  \overset{(I)}{\approx} [(0'' \to x) \to \{(x' \to x') \to 0' \}' ]' $  
\begin{align*}
	& \approx  [(0 \to x) \to \{(x' \to x') \to (0'' \to 0'') \}' ]'   
	\overset{hyp}{ \approx}  [(0 \to x) \to (0'' \to 0'')' ]' \\  
	& \approx  [(0 \to x) \to 0'' ]'  
	  \approx (0 \to x)''    
          \approx  0 \to x,   
	\end{align*}
proving the lemma.
\end{proof}

\begin{theorem} 
	Let $\mathbf{A} \in \C{I}_{2,0}$.  Then
	the following are equivalent in $\mathbf{A}$:
	\begin{enumerate} [{ \rm (1)}]
	\item
		 $\mathbf{A} \models (x \to x) \to (y \to y)  \approx y \to y$,
	\item  $\mathbf{A}$ is a Kleene algebra.
	\end{enumerate}
\end{theorem}

\begin{proof}	
(1) $\Rightarrow (2):$
 $(x \to y) \to x	 \overset{\ref{general_properties2} (\ref{291014_10})}{ \approx}  (0 \to y) \to x  
\overset{\ref{L:aux}}{\approx}  0'  \to x  
	\overset{\ref{general_properties2} (\ref{291014_10})}{\approx}   x''   	
	\overset{hyp}{\approx}   x$.  	
Hence $\mathbf{A}$ is a De Morgan algebra and so, in view of (1), $\mathbf{A}$ is a Kleene algebra.  Thus (1) implies (2).  The implication (2) $\Rightarrow$(1) follows directly from the definition of Kleene algebras.   	
\end{proof}

\section{A Glivenko-like theorem for $\C{I}$ }\label{S:5}

Let $\mathbf A \in \C{I}$. We introduce the following notation:
\[A'' = \{a'' \mid a \in A\}.\]
We also let $\mathbf{A''} := \langle A'', \to, 0 \rangle$.

\begin{lemma} \label{Lemma_App_subalg}
	Let $\mathbf A = \langle A, \to, 0 \rangle \in \C{I}$. Then
	$\mathbf{A''}$ is a subalgebra of $\mathbf{A}$.
\end{lemma}

\begin{proof}
	Since $0 = 0''$ we obtain that $0 \in A''$. If $a,b \in A''$, $a = a''_0$ and $b = b''_0$ with $a_0, b_0 \in A$. In view of Lemma \ref{Lemma_300315_01} (\ref{060415_01}), Lemma \ref{Lemma_300315_01} (\ref{300315_02}) and  \ref{SankaLemma7.5} we have that  
$(a \to b)'' = (a''_0 \to b''_0)'' =  ((a'_0 \to 0) \to b''_0)'' \overset{\ref{Lemma_300315_01} (\ref{060415_01})}{=} (a'_0 \to 0) \to b''_0 = a''_0 \to b''_0 =  a \to b$.
\end{proof}

We now present a Glivenko-like result for  algebras in $\C{I}$. In view of Definition \ref{definit_DM_KL_BA}, the
 following Theorem is immediate from Lemma \ref{Lemma_App_subalg}.

\begin{theorem} \label{Lemma_App_in_I20} 
	Let $\mathbf A = \langle A, \to, 0 \rangle \in \C I$, then
	\begin{enumerate}[{\rm(a)}] 
		\item  $\langle A'', \to, 0 \rangle \in \C{I}_{2,0}$,
		\item  $\mathbf A \models (x'' \to y'') \to x'' \approx x''$ iff $\mathbf{A''}$ is a De Morgan algebra,
		\item  $\mathbf A \models (y'' \to y'') \to (x'' \to x'') \approx x'' \to x''$ iff $\mathbf{A''}$ is a Kleene algebra,
		\item  $\mathbf A \models x'' \to x''  \approx 0'$ iff $\mathbf{A''}$ is a Boolean algebra.
	\end{enumerate}
\end{theorem}

\section{The varieties $\C{SCP} $ and $\C{MC}$}\label{S:6}

In this section, we describe the relationship between $\C{SCP} $ and $\C{MC}$.

\begin{lemma}\label{L:SCP}
	Let $\mathbf{A} \in \C{SCP}$.  Then $\mathbf{A}$ satisfies:
	\[0' \to [(x \to y) \to z] \approx (x \to y) \to z.\]
\end{lemma}

\begin{proof} 
\quad	$(x \to y) \to z \overset{(I)}{\approx}  [(x' \to y) \to (z \to x)']' $\\ 
	$\overset{(SCP)}{\approx}  0' \to [(x' \to y) \to (z \to x)']'   
	\overset{(I)}{ \approx}  0' \to [(x \to y) \to z] $. 
\end{proof}

\begin{theorem} \label{SCP_subset_MC}
$\C{SCP}  \subset \C{MC}$.	
\end{theorem}

\begin{proof}  \quad	$y \land x	\approx  (y \to x')'       
	\overset{(SCP)}{\approx}  (x'' \to y')'   
  \approx   [(x' \to 0'') \to y']'  $  
         \begin{align*}
\noindent &\overset{(SCP)}{\approx}  [(0' \to x) \to y']'     
	 \approx  [(0' \to y) \to (0' \to x)']'    
 \approx    [(0' \to x)'' \to (0' \to y)']' \\   
&\approx   [\{0' \to (0' \to x)\} \to (0' \to y)']'     
\overset{\ref{L:SCP}}{ \approx}    [(0' \to x) \to (0' \to y)']'  \\  
&\overset{(SCP)}{\approx}    [(0' \to x) \to (y' \to 0'')']'     
	\approx     [(0' \to x) \to (y' \to 0)']' \\	
&\overset{(I)}{\approx}  (x \to y')'      
	 \approx    x \land y. 
	  \end{align*}

The example in Figure~\ref{fig1} shows (at 2,2) that the inclusion is proper.
\begin{figure}[ht] 
\begin{center} 
\begin{tabular}{r|rrr}
$\to$: & $0$ & $1$ & $2$\\
\hline
    $0$ & $0$ & $0$ & $0$ \\
    $1$ & $0$ & $0$ & $0$ \\
    $2$ & $0$ & $0$ & $1$
\end{tabular} \hspace{.5cm}
\begin{tabular}{r|rrr}
$'$: & $0$ & $1$ & $2$\\
\hline
   & $0$ & $0$ & $0$
\end{tabular} \hspace{.5cm}
\begin{tabular}{r|rrr}
$\land$: & $0$ & $1$ & $2$\\
\hline
    $0$ & $0$ & $0$ & $0$ \\
    $1$ & $0$ & $0$ & $0$ \\
    $2$ & $0$ & $0$ & $0$
\end{tabular} 
\end{center}
\caption{The Example for Theorem~\ref{SCP_subset_MC}}\label{fig1}
\end{figure}
\end{proof}

A much stronger version of Theorem \ref{SCP_subset_MC} holds in $\C{I}_{2,0}$.  Indeed,   
$\C{SCP}$ and $\C{MC}$ will coincide in $\C{I}_{2,0}$.  

\begin{theorem} \label{Theo_300615_01}
	$\C{MC}  \cap  \C{I}_{2,0}  =  \C{SCP}  \cap  \C{I}_{2,0}$.
\end{theorem}

\begin{proof}
Let $\mathbf A \in \C{MC}\  \cap \ \C{I}_{2,0}$. Observe that $x \to y \approx (x \to y'')'' \approx  (x \land y')' \approx  (y' \land x)' \approx  (y' \to x')'' \approx  y' \to x'$. Hence $\C{MC}\  \cap \ \C{I}_{2,0}\  \subseteq \ \C{SCP} \  \cap \ \C{I}_{2,0}$.  The other inclusion is proved in the preceding theorem.
\end{proof}

\section{The algebra $\mathbf{A^{mj}}$\label{S:7} as a lattice} \label{lattice}

In this section we present necessary and sufficient conditions on algebras $\mathbf{A} \in \C{I}$ under which the {\it derived} algebra
$\mathbf{A^{mj}} = \langle A, \land, \lor, 0 \rangle$ is a lattice.
\begin{lemma} \label{130515_01}
	Let $\mathbf{A} \in \C{I}_{2,0}$. If $\mathbf{A}  \models x \land (x \lor y) \approx x$ then $\mathbf{A}  \models (x \to y) \to x \approx x$.
\end{lemma}

\begin{proof}
\quad $(x \to y) \to x \overset{\ref{general_properties2} (\ref{291014_10})} \approx  (0 \to y) \to x   
\approx (0 \to y)'' \to x  $   
\begin{align*}
&\overset{\ref{general_properties_equiv} (\ref{TXX})} {\approx}  [0' \to (0 \to y)']' \to x       
 \approx  (0' \land (0 \to y)) \to x                        
\approx  (0' \land (0'' \to y'')'') \to x \quad \mbox{} \\
& \approx  (0' \land (0'' \land y')') \to x      
\approx  (0' \land (0' \lor y)) \to x \quad   
\overset{hyp}{ \approx} 0' \to x     
\overset{\ref{general_properties_equiv} (\ref{TXX})}{ \approx}  x,
\end{align*}
 which completes the proof.	
\end{proof}

\begin{lemma}
	Let $\mathbf{A}  \in \C{I} $. If $\mathbf{A}  \models x \land (x \lor y) \approx x$ then $\mathbf{A}  \models (x \to y) \to x \approx x$.
\end{lemma}

\begin{proof}
	Let $a,b \in A$. Then, using (2) and (3) of  Lemma \ref{Lemma_300315_01} and Lemma \ref{SankaLemma7.5}, we have

	$(a \to b) \to a
		 =  [(a \to b) \to a]''   
	 =  [(a'' \to b'') \to a'']'' .$   
 Since $\mathbf{A}''$ is a subalgebra of $\mathbf{A} $ (see Lemma \ref{Lemma_App_subalg}),  
	we have that $\mathbf{A} '' \models x \land (x \lor y) \approx x$. Hence, using $a'',b'' \in A''$ we conclude that, by Lemma \ref{130515_01}, $(a \to b) \to a = [a'']'' = a''$.
	So, $(a \to b) \to a = a'' = [a \land (a \lor b)]'' = a \land (a \lor b) = a$ by   
	hypothesis 
	and
	Lemma \ref{Lemma_300315_01}. 
			
\end{proof}

\begin{theorem}
	The following are equivalent in  $\mathbf{A}  \in \C{I}$:
	\begin{enumerate}[{\rm(1)}] 
		\item
		$\mathbf{A^{mj}}$ is a lattice,
		\item
		Absorption law holds in  $\mathbf{A^{mj}}$,
		\item
		$(x \to y) \to x \approx x$ holds in $\mathbf{A}$,
		\item
		$\mathbf{A}$ is a De Morgan algebra.
	\end{enumerate}
\end{theorem}

\begin{proof}
(1) $\Rightarrow$ (2) is trivial,  (2) $\Rightarrow$ (3) follows from the preceding lemma, and  (3) $\Rightarrow$ (4) follows from Definition 2.2, while (4) $\Rightarrow$ (1) was proved in 
\cite{sankappanavarMorgan2012}.
\end{proof}

Here is another interesting (surprising to us) characterization of De Morgan algebras.
But first we need a lemma.

\begin{lemma} \label{L7.1}
	Let $\mathbf{A} \in \C{I}_{2,0}$ such that
	$\mathbf{A} \models x \land 0 \approx 0$.
Then $\mathbf{A} \models x \to 0' \approx 0'$.  
\end{lemma}

\begin{proof}  Since $\mathbf{A} \in \C{I}_{2,0}$, we have                  
		$x \to 0'    \approx  (x \to 0')''      
		 \approx    (x \land 0)' $, implying   
	$x \to 0'  \approx   0' $,  
	since $\mathbf{A} \models x \land 0 \approx 0$.  
\end{proof}

Now we are ready to give the new characterization of De Morgan algebras.
\begin{theorem}
	Let $\mathbf{A} \in \C{I}$.  Then the following are equivalent:
	\begin{enumerate}[{\rm (1)}] 
		\item
		\begin{enumerate}[{\rm(a)}] 
			\item
			$\mathbf{A} \models x'' \approx x$,
			\item
			$\mathbf{A} \models x \land 0 \approx 0$,
		\end{enumerate}
		\item 
		$\mathbf{A} \models (0 \to x) \to y \approx y$,	
		\item
		$\mathbf{A}$ is a De Morgan algebra.
	\end{enumerate}
\end{theorem}

\begin{proof} 
Suppose (1) holds.  Then   $(0 \to x) \to y   \approx  (0 \to x'') \to y  
\overset{\ref{general_properties} (2)}{\approx}   (x' \to 0') \to y  
	                          \approx  0' \to y  
	                          \approx  y $.      
Hence (1) implies (2).  The other implications follow from \cite{sankappanavarMorgan2012}.	
\end{proof}

\section{ Some properties of the varieties $\C{I}$ and  $\C{I}_{2,0} \cap  \C{MC}$}\label{S:8}

In this section we present more properties of  algebras in $\C{I}$ and also some properties of algebras in $\C{I}_{2,0}  \cap  \C{MC}$.

\begin{theorem}\label{T8.1}
	Let  $\mathbf A \in \C{I}$.  Then
	\begin{enumerate}[{\rm(a)}] 
		\item  $\mathbf{A^{mj}}$ satisfies:
		\begin{enumerate}[{\rm(1)}] 
			\item  $ (x  \lor  y)'  \approx x'  \land y'$,
			\item  $ (x \land y)'  \approx  x' \lor  y' $
		\end{enumerate}
		\item The following are equivalent in $\mathbf{A^{mj}}$:
		\begin{enumerate}[{\rm(1)}]  
			\item  $x \land y \approx y \land x$  \quad {\rm 
			(i.e.,
			$\land$-commutative)},
			\item   $x \lor y \approx y \lor x$ \quad {\rm 
			(i.e., 
			$\lor$-commutative)},
		\end{enumerate}
		\item  The following are equivalent in $\mathbf{A^{mj}}$:
		\begin{enumerate}[{\rm(1)}]  
			\item  $x \land (y \lor z) \approx (x \land y) \lor (x \land z)$,
			\item  $x \lor (y \land z) \approx (x \lor y) \land (x \lor z)$.
		\end{enumerate}
	\end{enumerate}
\end{theorem}

\begin{proof}

		(a): (1) $\Rightarrow$ (2):  By Lemma \ref{Lemma_300315_01} (\ref{300315_01}) we have that $(x \lor y)' \approx  (x' \land y')'' \approx  (x' \to y'')''' \approx  [(x \to 0) \to y'']''' \approx  [(x \to 0) \to y'']' \approx   (x' \to y'')' \approx  x' \land y'$.
			
		 (2) $\Rightarrow$ (1):  From Lemma \ref{SankaLemma7.5}  and Lemma \ref{Lemma_300315_01} (\ref{300315_02}), $x' \lor y' \approx  (x'' \land y'')' \approx  (x'' \to y''')'' \approx  (x \to y''')'' \approx  (x \to y')'' \approx  (x \land y)'$.
		
		(b):  	
		From (1) we have that $x \lor y \approx  (x' \land y')' \approx  (y' \land x')' \approx  y \lor x$, implying (y).
		Suppose (2) holds.  Then by Lemma \ref{Lemma_300315_01} (\ref{310315_04}), $x \land y \approx  (x' \lor y')' \approx  (y' \lor x')' \approx  y \land x$, yielding (1).

(c):			
		 (1) $\Rightarrow$ (2):
\quad $x \lor (y \land z)   \overset{\text{def of } \lor}{\approx}  [x' \land (y \land z)']'  
\overset{\ref{Lemma_300315_01} (\ref{310315_04})}{\approx}  [x' \land (y' \lor z')'']' $ 
\begin{align*}
&\overset{\ref{Lemma_300315_01} (\ref{310315_05})} { \approx}  [x' \land (y' \lor z')]'   
\overset{hyp}{\approx}  [(x' \land y') \lor (x' \land z')]'  
\overset{\ref{Lemma_300315_01} (\ref{310315_03})} {\approx}  [(x' \land y')'' \lor (x' \land z')'']' \\ 
&\overset{\text{def of } \lor}{\approx}  [(x \lor y)' \lor (x \lor z)']'   
\overset{\ref{Lemma_300315_01} (\ref{310315_04})}{\approx}  (x \lor y) \land (x \lor z).  
\end{align*}

(d):  (2) $\Rightarrow$ (1):  
\quad		$x \land (y \lor z)	 \overset{\ref{Lemma_300315_01} (\ref{310315_04})} {\approx}  [x' \lor (y \lor z)']'   
\overset{\text{def of } \lor}{\approx}  [x' \lor (y' \land z')'']' $  
\begin{align*}
&\overset{\ref{Lemma_300315_01} (\ref{310315_03})} {\approx}  [x' \lor (y' \land z')]'  
\overset{hyp}{\approx}  [(x' \lor y') \land (x' \lor z')]'  
\overset{\ref{Lemma_300315_01} (\ref{310315_05})}{\approx}  [(x' \lor y')'' \land (x' \lor z')'']' \\  
&\overset{\ref{Lemma_300315_01} (\ref{310315_04})}{\approx}  [(x \land y)' \land (x \land z)']'   
\overset{\text{def of } \lor}{\approx}  (x \land y) \lor (x \land z).  
\end{align*}	
This completes the proof.
\end{proof}

\begin{theorem} Let $\mathbf{A} \in \C{I}$.  Then the following hold in $\mathbf{A^{mj}}$:
	\begin{enumerate}[{\rm(a)}]  
		\item 	$(x \land y) \land (x \lor y) \approx x \land  y$, \label{080415_04}
		\item    $(x \lor y)  \lor  (x \land y) \approx x \lor  y$. \label{080415_05}
	\end{enumerate}	
\end{theorem}

\begin{proof} (a):
\quad  $(x \land y) \land (x \lor y) \overset{\text{def of } \lor}{\approx}  (x \land y)  
\overset{\text{def of } \land}{\approx}  (x \to y')' \land (x' \to y'')'' $ 
\begin{align*}
&\overset{\text{def of } \land}{\approx}  [(x \to y')' \to (x' \to y'')''']'  
\overset{\ref{Lemma_300315_01} (\ref{300315_01})}{\approx}  [(x \to y')' \to (x' \to y'')']' \\ 
&\overset{\ref{Lemma_300315_01} (\ref{300315_02})}{\approx}  [(x'' \to y')' \to (x''' \to y'')']' 
\overset{\ref{SankaLemma7.5}}{\approx}  [(x'' \to y''')' \to (x''' \to y'')']'.    
\end{align*}	
Notice that, by Lemma \ref{SankaTheo7.6}, $x'', y'' \in A''$. Hence by Theorem \ref{Lemma_App_in_I20} and Lemma \ref{general_properties2} (\ref{080415_03}), we have that $[(x'' \to y''')' \to (x''' \to y'')']' \approx (x'' \to y''')'$. Then by Lemma \ref{Lemma_300315_01} (\ref{300315_02}) and  Lemma \ref{SankaLemma7.5}, we obtain 
	$(x \land y) \land (x \lor y) \approx  (x \to y')' \approx  x \land y$. 	
	
(b):	
\quad $(x \lor y)  \lor  (x \land y) \overset{\text{def of } \lor}{\approx}  (x' \land y')' \lor (x \land y) $
\begin{align*}
\noindent&\overset{\text{def of } \lor}{\approx} [(x' \land y')'' \land (x \land y)']'   
\overset{\ref{Lemma_300315_01} (\ref{310315_03})}{\approx}  [(x' \land y') \land (x \land y)']' \\ 	
&\overset{\ref{Lemma_300315_01} (\ref{310315_04})}{\approx}  [(x' \land y') \land (x' \lor y')'']'   
	\approx [(x' \land y') \land (x' \lor y')]' 		 
\overset{(\ref{080415_04})}{\approx}  [x' \land y']' 
	\approx  x \lor y,  
\end{align*} 
completing the proof.
\end{proof}

We should note here that the conclusions in the preceding theorem are known to hold in every Birkhoff system (see \cite{harding2012proyective}).

Recall that ${\C MC}$ denotes the subvariety of $\C{I}$ defined by the identity: $x \land y \approx y \land x$.  Next, we consider the variety $\C{I}_{2,0}  \cap  \C{MC}$ and show that it has some interesting properties.

\begin{theorem}\label{T:7.3}
	Let $\mathbf A \in \C{I}_{2,0}  \cap  \C{MC}$.  Then $\mathbf{A^{mj}}$ satisfies:
	\begin{enumerate}[{\rm(a)}]  
		\item  $x \land x \approx x$, \label{010415_01}
		\item  $x \lor x \approx x $,
		\item  $x \lor y \approx y \lor x$,
		\item  $x \land (y  \lor z)\approx (x  \land y) \lor (x \land z)$,  \label{010415_02}
		\item  $x \lor (y  \land z)\approx (x  \lor y) \land (x \lor z)$, 
		\item  $x \land (x \lor y) \approx x \lor (x \land y)$. \label{010415_03} 
	\end{enumerate}
\end{theorem}

\begin{proof}
	(a) and (b) follow from Theorem \ref{T3.5} and (c) follows from Theorem \ref{T8.1}(2).  We will prove (d) and (f).  Then (e) will follow from (d) and Theorem \ref{T8.1}(3).

(d):	
\quad   $(x  \land y) \lor (x \land z)  \approx  [(x  \land y)' \land (x \land z)' ]'  $ 
\begin{align*}
&\overset{\text{def of } \lor}{\approx}  [(x  \land y)' \to (x \land z)'']'' 
\overset{\ref{Lemma_300315_01} (\ref{310315_03})}{\approx}  [(x  \land y)' \to (x \land z)]'' \\ 
&\overset{hyp}{\approx}  [(x  \land y)' \to (z \land x)]''   
\overset{\text{def of } \land}{\approx}  [(x  \to  y')'' \to (z \to x')']'' \\  
&\overset{\ref{Lemma_300315_01} (\ref{300315_02})}{\approx}  [(x  \to  y') \to (z \to x')']''  
\overset{(I)}{\approx}  [(y' \to z) \to x]' \\  
&\overset{\text{def of } \land}{\approx}  (y' \to z) \land x            
\overset{hyp}{\approx}  x \land (y' \to z)\\    
&\overset{\text{def of } \land}{\approx}  [x \to (y' \to z)']'   
\overset{\ref{SankaLemma7.5}}{\approx}  [x \to (y' \to z'')']'   
\overset{\text{def of } \land}{\approx}  x \land (y' \to z'') \\
&\overset{(\ref{010415_01})}{\approx}  x \land (y' \to z'')''  
\overset{\text{def of } \land}{\approx}  x \land [y' \land z']'  
\overset{\text{def of } \lor}{\approx}  x \land (y \lor z).  
\end{align*}

(f):	
\quad	$x \lor (x \land y) \overset{\text{def of } \lor}{\approx}  [x' \land (x \land y)']'  
\noindent \overset{\text{def of } \land}{\approx} [x' \to (x \land y)'']'' $  \begin{align*}
&\overset{(\ref{010415_01})}{\approx}  [x' \to (x \land y)]   
\overset{hyp}{\approx}  [x' \to (y \land x)]  
\overset{\text{def of } \land}{\approx} [x' \to (y \to x')'] \\ 
&\approx  [(x \to 0) \to (y \to x')']  
\overset{(\ref{010415_01})}{\approx}  [(x'' \to 0) \to (y \to x')'] \\ 
&\overset{(\ref{010415_01})}{\approx}  [(x'' \to 0) \to (y \to x')']''   
\overset{(I)}{\approx}  [(0 \to y) \to x']'  
\overset{\ref{general_properties} (\ref{cuasiConmutativeOfImplic2})}{\approx}  [(y' \to 0') \to x']' \\  
&\overset{\ref{general_properties2} (\ref{071114_02})}{\approx}  (0 \to y') \to x''   
\overset{(\ref{010415_01})}{\approx}  (0 \to y') \to x    
\overset{\ref{general_properties2} (\ref{291014_10})}{\approx}  (x \to y') \to x \\   
&\overset{\ref{general_properties2} (\ref{291014_09})}{\approx}  [x \to (y' \to x)']'   
\overset{ (\ref{010415_01})}{\approx}  [x \to (y' \to x'')']'   
\overset{\text{def of } \land}{\approx}  [x \to (y' \wedge x')]' \\   
&\overset{hyp}{\approx}  [x \to (x' \wedge y')]'   
\overset{\ref{Lemma_300315_01} (\ref{310315_03})}{\approx}  [x \to (x' \wedge y')'']' \\  
&\overset{\text{def of } \lor}{\approx}  [x \to (x \vee y)']'   
	 \approx  x \land (x \vee y).  
\end{align*}
This completes the proof.
\end{proof}

 We should mention here that the absorption identity $x \land (x \lor y) \approx x$  fails in $\C{I}_{2,0} \cap \C{MC}$; so, if $\mathbf{A} \in  \C{I}_{2,0} \cap \C{MC}$, then $\mathbf{A^{mj}}$ is not a lattice.  Indeed, in Section \ref{S:7}, we have already shown that $\mathbf{A^{mj}}$ is a lattice precisely when $\mathbf {A}$ is a De Morgan algebra.

\begin{remark}
	We note here that $\C{I}_{2,0}$ {\rm (= $\C{MID}$)} and $\C{MC}$ are incomparable in the lattice of subvarieties of $\C{I}$ as the 
examples in Figures~\ref{fig2} and~\ref{fig3} show	
\end {remark}
	
\begin{figure}[ht]	

	\begin{center} 	
		\begin{tabular}{r|rr}
			$\to$: & $0$ & $1$\\
			\hline
			$0$ & $0$ & $0$ \\
			$1$ & $0$ & $0$
		\end{tabular} \hspace{.5cm}
		\begin{tabular}{r|rr}
			$'$: & $0$ & $1$\\
			\hline
			& $0$ & $0$
		\end{tabular} \hspace{.5cm}
		\begin{tabular}{r|rr}
			$\land$: & $0$ & $1$\\
			\hline
			$0$ & $0$ & $0$ \\
			$1$ & $0$ & $0$
		\end{tabular}    
\caption{$\C{MC} \not \subseteq \C{MID}$}\label{fig2}
	\end{center}
\end{figure}
\begin{figure}[ht]	
 
	\begin{center} 
	\begin{tabular}{r|rrr}
			$\to$: & $0$ & $1$ & $2$\\
			\hline
			$0$ & $0$ & $1$ & $2$ \\
			$1$ & $1$ & $1$ & $2$ \\
			$2$ & $2$ & $1$ & $2$
		\end{tabular} \hspace{.5cm}
		\begin{tabular}{r|rrr}
			$'$: & $0$ & $1$ & $2$\\
			\hline
			& $0$ & $1$ & $2$
		\end{tabular} \hspace{.5cm}
		\begin{tabular}{r|rrr}
			$\land$: & $0$ & $1$ & $2$\\
			\hline
			$0$ & $0$ & $1$ & $2$ \\
			$1$ & $1$ & $1$ & $2$ \\
			$2$ & $2$ & $1$ & $2$
		\end{tabular}   
	\end{center}
\caption{$\C{MID} \not \subseteq \C{MC}$}\label{fig3}	
\end{figure}

\section{The varieties $\C{Z}$, $\C{C}$, ${\C A}$, and $\C{I}_{3,1}$ }\label{S:9}

In this section our goal is to prove  $\C{Z} \subset \C{C} \subset {\C A} \subset \C{I}_{3,1}$.  For this purpose, we first need to establish the following lemma.

\begin{lemma} \label{properties_commut_zrupoids}
	Let $\mathbf{A} \in \C{C}$. Then $\mathbf A$ satisfies the identities:
	\begin{enumerate}[{\rm(a)}]   
		\item   $0 \approx 0'$, \label{060415_02}
		\item   $x' \approx (x \to x)'$, \label{060415_03}
		\item   $(x \to y)' \approx (x' \to y)' \approx (x' \to y')' \approx (x \to y')'$,  \label{060415_05}
		\item   $x \to (x \to y) \approx y' \to x$. \label{060415_04}
	\end{enumerate}
\end{lemma}

\begin{proof}
	
	(\ref{060415_02}):
	 	 First we see that $0 = 0'' = 0' \to 0 = 0 \to 0'$ by (C).  So, $0 = 0 \to 0' = 0'' \to 0' = (0' \to 0) \to 0' = [(0'' \to 0') \to (0 \to 0')']' = [(0 \to 0') \to (0 \to 0')']' = [0 \to 0']' = 0'$.
		
		(\ref{060415_03}):
	$x'	 \approx  x \to 0  
	\overset{(C)}{\approx}  0 \to x  
\overset{(\ref{060415_02})}{ \approx}  (0 \to 0) \to x  
\overset{(I)}{\approx}  [(x' \to 0) \to (0 \to x)']'  
		 \overset{(C)}{\approx} [(x' \to 0) \to (x \to 0)']'  
	          \approx  (x'' \to x'')'  
\overset{\ref{Lemma_300315_01} (\ref{300315_02})}{\approx}  (x \to x'')'  
\overset{\ref{SankaLemma7.5}}{\approx}  (x \to x)'. $

	 	(\ref{060415_05}):
	$	(x \to y)'	 \approx  (x \to y) \to 0 
		 \approx  [(0' \to x) \to (y \to 0)']'  
		 \approx   [(x \to 0') \to (y \to 0)']'  
		 \approx  [(x \to 0') \to y'']'   
\overset{(\ref{060415_02})}{\approx}  [(x \to 0) \to y'']'   
		 \approx  (x' \to y'')'    
\overset{\ref{SankaLemma7.5}}{\approx}  (x' \to y)' .$  
		Hence, it follows that $(x \to y)' \approx (x' \to y')' \approx (x \to y')'$.

           	 (\ref{060415_04}):
		$x \to (x \to y)  \overset{(C)}{\approx}  (x \to y) \to x  
		\overset{(I)}{\approx} [(x' \to x) \to (y \to x)']'  
\overset{(\ref{Lemma_300315_01} (\ref{300315_02})}{\approx}  [(x' \to x)'' \to (y \to x)']'   
		\overset{(\ref{060415_05})}{\approx}   [(x \to x)'' \to (y \to x)']'  
\overset{(\ref{060415_03})}{\approx}	   [x'' \to (y \to x)']'   
		 \approx  [(x' \to 0) \to (y \to x)']'   
		 \approx (0 \to y) \to x   
		 \approx  (y \to 0) \to x   
		 \approx  y' \to x. $ 	
\end{proof}

Our next theorem describes the relationship among the varieties  $\C{Z}$, $\C{C}$, ${\C A}$ and $\C{I}_{3,1}$.

\begin{theorem} \label{Theorem_C_sub_Assoc}
	$\C{Z} \subset \C{C} \subset {\C A} \subset \C{I}_{3,1}$.
\end{theorem}

\begin{proof}
It is clear that $\C{Z} \subset \C{C} $.   So, let $\mathbf{A} \in \C{C}$. 
	Notice that 
      
	$(x \to y) \to z	 \approx  [(z' \to x) \to (y \to z)']'  
\overset{\ref{Lemma_300315_01} (\ref{300315_02})}{\approx}  [(z' \to x)'' \to (y \to z)']' $ 
\begin{align*}
&\overset{\ref{properties_commut_zrupoids} (\ref{060415_05})}{\approx}  
[(z \to x)'' \to (y \to z)']' 
\overset{\ref{Lemma_300315_01} (\ref{300315_02})}{\approx}  [(z \to x) \to (y \to z)']' \\ 
&\overset{\ref{properties_commut_zrupoids} (\ref{060415_05})}{\approx}  [(z \to x) \to (y \to z)]' 
	 \approx  [(y \to z) \to (z \to x)]' \\  
&\overset{(I)}{\approx}  [\{(z \to x)' \to y\} \to \{z \to (z \to x)\}']'' \\ 
&\overset{\ref{Lemma_300315_01} (\ref{300315_02})}{\approx}  
[\{(z \to x)' \to y\}'' \to \{z \to (z \to x)\}']'' \\  
&\overset{\ref{properties_commut_zrupoids} (\ref{060415_05})}{\approx}  
  [\{(z \to x) \to y\}'' \to \{z \to (z \to x)\}']'' \\
&\overset{\ref {Lemma_300315_01} (\ref{300315_02})}{\approx}  
[\{(z \to x) \to y\} \to \{z \to (z \to x)\}']'' \\ 
&\overset{\ref{Lemma_300315_01} (\ref{300315_02})}{\approx}  [\{(z \to x) \to y\} \to \{z \to (z \to x)\}]'' 
\overset{\ref{properties_commut_zrupoids}(\ref{060415_04})}{\approx}  [\{(z \to x) \to y\} \to (x' \to z)]'' \\
&\overset{\ref{SankaLemma7.5}}{\approx}  [\{(z \to x) \to y\} \to (x' \to z)'']'' 
\overset{\ref{properties_commut_zrupoids} (\ref{060415_05})}{\approx}  [\{(z \to x) \to y\} \to (x \to z)'']'' \\
&\overset{\ref{SankaLemma7.5}}{\approx}  [\{(z \to x) \to y\} \to (x \to z)]''  
\overset{\ref{Lemma_300315_01} (\ref{060415_01})}{\approx}  [(z \to x) \to y] \to (x \to z). 
\end{align*}     
	
	Also,
	
	$x \to (y \to z)	 \approx (y \to z) \to x  
	 \approx  [(x' \to y) \to (z \to x)']' $ 
\begin{align*}
&\overset{\ref{Lemma_300315_01} (\ref{300315_02})}{\approx}  [(x' \to y)'' \to (z \to x)']'  
\overset{\ref{properties_commut_zrupoids} (\ref{060415_05})}{\approx}  
  [(x \to y)'' \to (z \to x)']' \\ 
&\overset{\ref{Lemma_300315_01} (\ref{300315_02})}{\approx}  [(x \to y) \to (z \to x)']'  
\overset{\ref{properties_commut_zrupoids} (\ref{060415_05})}{\approx}  [(x \to y) \to (z \to x)]' \\ 
&\overset{(I)}{\approx}  [\{(z \to x)' \to x\} \to \{y \to (z \to x)\}']'' \\
&\overset{\ref{Lemma_300315_01} (\ref{300315_02})}{\approx}  [\{(z \to x)' \to x\}'' \to \{y \to (z \to x)\}']'' \\ 
&\overset{\ref{properties_commut_zrupoids} (\ref{060415_05})}{\approx} [\{(z \to x) \to x\}'' \to \{y \to (z \to x)\}']'' \\ 
&\overset{\ref{Lemma_300315_01} (\ref{300315_02})}{\approx}  [\{(z \to x) \to x\} \to \{y \to (z \to x)\}']'' \\ 
&\overset{\ref{properties_commut_zrupoids} (\ref{060415_05})}{\approx}  [\{(z \to x) \to x\} \to \{y \to (z \to x)\}]'' \\
&\approx  [\{x \to (x \to z)\} \to \{y \to (z \to x)\}]''  
\overset{\ref{properties_commut_zrupoids} (\ref{060415_04})}{\approx}  [(z' \to x) \to [y \to (z \to x)]]'' \\ 
&\overset{\ref{Lemma_300315_01}(\ref{300315_02})}{\approx}  [(z' \to x)'' \to [y \to (z \to x)]]''  
\overset{\ref{properties_commut_zrupoids} (\ref{060415_05})}{\approx}  [(z \to x)'' \to [y \to (z \to x)]]'' \\ 
&\overset{\ref{Lemma_300315_01} (\ref{300315_02})}{\approx}  [(z \to x) \to [y \to (z \to x)]]''  
	 \approx  [\{(z \to x) \to y\} \to (x \to z)]'' \\ 
& \approx  [(z \to x) \to y] \to (x \to z).  
\end{align*}	
	
	Hence,  $(x \to y) \to z \approx x \to (y \to z)$.
	For the second half, let $\mathbf{A} \in {\C A}$. 
	$x'''	 \approx  [(x \to 0) \to 0]'  
\overset{(A)}{\approx}  [x \to(0 \to 0]'  
\approx (x \to 0') \to 0 
\overset{(A)}{\approx}  x \to (0' \to 0) 
	 \approx x \to 0'' 
	 \approx  x \to 0 
	 \approx  x'. $  
	
	Finally, Example 2 of Section \ref{S:9} shows that the first two inclusions are strict, while the 2-element Boolean algebra illustrates the strictness of the third inclusion.	
\end{proof}

\section{The varieties $\C{I}_{1,0} $,   $\C{ID}$ and ${\C A}$}\label{S:10}

We recall from \cite{sankappanavarMorgan2012} that $\C{I}_{1,0}$ and $\C{ID}$
denote the subvarieties of $\C{I}$ defined, respectively, by:
\begin{gather*}
x' \approx x, \\
x \to x \approx x.
\end{gather*}
(The variety $\C{ID}$ was actually called $\mathbf{IDMP}$ in \cite{sankappanavarMorgan2012}.)

In this section we wish to show that $\C{I}_{1,0} =  \C{ID} \cap  {\C A}$.  For this we need the following two lemmas.

\begin{lemma} \label{Propierties_I10}
	Let $\mathbf{A} \in \C{I}_{1,0}$.   
	Then the following are true in $\mathbf{A}$:
	\begin{enumerate}[{\rm(a)}]  
		\item   $(y \to z) \to x \approx  (x \to y) \to (z \to x)$, \label{010715_15}
		\item   $(0 \to x) \to y \approx x \to y$, \label{010715_16} 
		\item   $x \to (y \to x) \approx (0 \to y) \to x $,  \label{010715_17} 
		\item   $x \to (y \to x) \approx y \to x$. \label{010715_18}  
          	\end{enumerate}
\end{lemma}

\begin{proof} Let $a,b \in A$.
	(\ref{010715_15}) is immediate from (I) and $a' \approx a$.   Since $\mathbf A \in \C{I}_{1,0}$, we get $a \to b = (a \to b)' = [(0' \to a) \to b'']' =(0 \to a) \to b$, proving (\ref{010715_16}), and the proof of (\ref{010715_17}) is similar, while (\ref{010715_18}) follows from (\ref{010715_16}) and (\ref{010715_17}).
	Since $\C{I}_{1,0} \subseteq \C{I}_{2,0}$, using Lemma \ref{general_properties2} (\ref{291014_10}) and Lemma \ref{general_properties} (\ref{cuasiConmutativeOfImplic}), we have that $(a \to b) \to a = (0 \to b) \to a = (b' \to 0') \to a = (b \to 0) \to a = b \to a$. Hence (\ref{010715_18}) holds in $\mathbf{A}$.
\end{proof}

\begin{lemma}\label{IDAssoc}	
	
	Let $\mathbf{A} \in \C{ID}\ \cap \ {\C A}$. 
	Then the following are true in $\mathbf{A}$:
	\begin{enumerate}[{\rm(1)}]  
		\item  $y \to (z \to x) \approx x \to [0 \to \{y \to (z \to x')\}]$, \label{010715_09} 
		\item  $x \to (y \to x') \approx y \to x$,  \label{010715_10} 
		\item  $x \to (0 \to x) \approx x'$,  \label{010715_11} 
		\item  $(x \to y) \to [0 \to \{x \to (y \to (x \to y'))\}] \approx x \to y $,  \label{010715_12} 
		\item  $(x \to y) \to [0 \to  (x \to y)] \approx x \to y$,  \label{010715_13} 
		\item  $x \to y' \approx x \to y$. \label{IDAssoc_39}  \label{010715_14} 
	\end{enumerate}
\end{lemma}
\begin{proof}

		(\ref{010715_09}): 
	$	y \to (z \to x)  \overset{(A)}{\approx}   (y \to z) \to x  
\overset{(I)}{\approx}  [(x'  \to y)  \to (z \to x)']' $  
\begin{align*}
&\overset{(A)}{\approx}   [\{x \to (0 \to y) \} \to (z \to x)']'  
	  \overset{(A)}{\approx}   [\{x \to (0 \to y)\} \to (z \to x')]' \\ 
&\overset{(A)}{\approx}   [x \to \{(0 \to y)  \to (z \to x')\}]'  
	 \overset{(A)}{\approx}   [x \to \{0 \to (y  \to (z \to x'))\}]' \\ 
&\overset{(A)}{\approx}   x \to [\{0 \to (y  \to (z \to x'))\} \to 0] \\ 
&\overset{(A)}{\approx}   x \to [0 \to \{(y  \to (z \to x')) \to 0\}] \\ 		
&\overset{(A)}{\approx}   x \to [0 \to \{y  \to ((z \to x') \to 0)\}] \\ 
&\overset{(A)}{\approx}   x \to [0 \to \{y  \to (z \to (x' \to 0))\}]\\   
&\overset{(A)}{\approx}   x \to [0 \to \{y  \to (x \to (x \to (0 \to 0)))\}] \\ 
&\overset{(ID)}{\approx}   x \to [0 \to \{y  \to (z \to x' )\}].\\    
\end{align*}

\ \\ \ \\  
			
			(\ref{010715_10}):

\quad $x \to (y \to x')  \overset{(\ref{010715_09})}{ \approx}  x \to [0 \to \{0 \to (y \to (x \to 0'))\}] $ 
\begin{align*}
&\overset{(ID)}{\approx}  x \to [0 \to \{0 \to (y \to x')\}]  
\overset{(A)}{\approx}   x \to [(0 \to 0) \to (y \to x')] \\ 
& \overset{(ID)}{\approx}   x \to [0 \to (y  \to x')]  
\overset{(ID)}{\approx}   x \to [0 \to \{y  \to (x \to x)'\}] \\ 
&\overset{(A)}{\approx}  x \to [0 \to \{y  \to (x \to x')\}] 
\overset{(\ref{010715_09})}{ \approx}  y \to (x \to x)  
	\overset{(ID)}{\approx}   y \to x.    
\end{align*}

		(\ref{010715_11}): 

	$ x \to (0 \to x)   
		 \overset{(A)}{\approx}    (x \to 0) \to x   
		 \overset{(I)}{\approx}   [(x' \to x) \to (0 \to x)']' $ 
\begin{align*}		 
&\overset{(A)}{\approx}   [\{x \to (0 \to x)\} \to (0 \to x)']'  
\overset{(A)}{\approx}   [x \to \{(0 \to x)  \to (0 \to x')\}]' \\ 
&\overset{(A)}{\approx}   [x \to (0 \to \{x \to (0 \to x')\}]'  
		 \overset{(\ref{010715_09}))}{\approx}   [x \to (0 \to x)]' 
		 \overset{(A)}{\approx}  x \to (0 \to x') \\ 
&\overset{(A)}{\approx}   x'  \to x' 
		 \overset{(ID)}{\approx}   x' .  
\end{align*}

		(\ref{010715_12}):  
		$(x \to y) \to [0 \to \{x \to (y \to (x \to y'))\}] $ 
\begin{align*}
&\overset{(A)}{\approx}     (x \to y) \to [0 \to \{x \to (y \to (x \to y)')\}]          
		  \overset{(\ref{010715_09})}{\approx}   x \to [y  \to (x \to y)] \\  
&\overset{(A)}{\approx}    (x \to y)  \to (x \to y)   
		  \overset{(ID)}{\approx}   x \to y.   	
\end{align*}

		(\ref{010715_13}): 
	$	(x \to y) \to [0 \to  (x \to y)] 
		 \overset{(ID)}{\approx}   (x \to y) \to [0 \to \{(x \to  x) \to y\}] $   
\begin{align*}
&\overset{(A)}{\approx}   (x \to y) \to [0 \to \{x \to  (x \to y)\}] \\ 
&\overset{(\ref{010715_10})}{\approx}  (x \to y) \to [0 \to \{x \to (y \to  (x \to y'))\}]  
\overset{(\ref{010715_12})}{\approx}    x \to y.   
\end{align*}

		(\ref{010715_14}):  

	$	x \to y  \overset{(\ref{010715_13})}{\approx}  (x \to y) \to [0 \to  (x \to y)]  
		  \overset{(\ref{010715_11})}{\approx}  (x \to y)'   
		 \overset{(A)}{\approx}   x \to y'.  $ 
\end{proof}

We are now ready to prove the theorem.
\begin{theorem} \label{Theo_I10_IDcapAssoc}
	$\C{I}_{1,0} =  \C{ID} \cap  {\C A}$.  
\end{theorem}

\begin{proof}
	Let $\mathbf{A} \in \C{I}_{1,0}$. 
Then, using (\ref{010715_16}) and (\ref{010715_18}) of Lemma \ref{Propierties_I10}, we have $x  \approx x' \approx 0 \to x \approx x \to (0 \to x) \approx x \to x$.  Thus, $\C{I}_{1,0} \subseteq \ \C{ID}$.  Observe that, since $\C{I}_{1,0} \subseteq \C{I}_{2,0}$, we get
\quad $(x \to y) \to z	\overset{\ref{Propierties_I10} (\ref{010715_15})}{\approx}  (z \to x) \to (y \to z)$ 
\begin{align*}
&\overset{\ref{Propierties_I10} (\ref{010715_15})}{\approx}  [(y \to z) \to z] \to [x \to (y \to z)] \\ 
&\overset{hyp}{\approx}  [(y \to z') \to z] \to [x \to (y \to z)]  
\overset{\ref{general_properties2} (\ref{281014_05})}{\approx} 
 [(y \to 0') \to z] \to [a \to (y \to z)] \\ 
&\overset{hyp}{\approx}  [(y \to 0) \to z] \to [x \to (y \to z)]   
\overset{hyp}{\approx}  (y \to z) \to [x \to (y \to z)] \\ 
&\overset{hyp}{\approx}  (y \to z) \to [x \to (y \to z)']  
\overset{\ref{general_properties2} (\ref{281114_01})}{\approx}  x \to (y \to z)'  
\overset{hyp}{\approx}  x \to (y \to z). 
\end{align*}
Hence $\C{I}_{1,0} \subseteq  \C{ID} \cap {\C A}$. 		
From	$x'   \overset{(ID)}{\approx}  (x \to x)'  
	 \overset{(A)}{\approx}  x \to x'  
	\overset{\ref{IDAssoc} (\ref{IDAssoc_39})}{\approx}  x  \to x	 
	 \approx  x,$  the other inclusion follows. 	 
\end{proof}

\section{Axiomatization of the varieties generated by the 2-element algebras in $\C{I}$}\label{S:11}

In this section we axiomatize the varieties generated by 
$2$-element algebras in $\C{I}$.
Recall from \cite{sankappanavarMorgan2012} that there are three $2$-element algebras in $\C{I}$, namely $\mathbf{2_z}$, $\mathbf{2_s}$, $\mathbf{2_b}$, whose $\to$
 operations are respectively as follows:   

\begin{figure}[ht] 
\begin{center} 
\begin{tabular}{r|rr}
	$\to$: & $0$ & $1$\\
	\hline
	$0$ & $0$ & $0$ \\
	$1$ & $0$ & $0$
\end{tabular} \hspace{.5cm} 
\begin{tabular}{r|rr}
	$\to$: & $0$ & $1$\\
	\hline
	$0$ & $0$ & $1$ \\
	$1$ & $1$ & $1$
\end{tabular} \hspace{.5cm}	
\begin{tabular}{r|rr}
		$\to$: & $0$ & $1$\\
		\hline
		$0$ & $1$ & $1$ \\
		$1$ & $0$ & $1$
	\end{tabular}    
\end{center}
\caption{$\to$ operations in the $2$-element algebras of $\C{I}$} \label{fig 4}
\end{figure}

\subsection{$\lor$-semilattices with the least element $0$ as a subvariety of $\C{I}$}

We first determine those algebras of $\C{I}$ which are $\lor$-semilattices with the least element $0$.

Let ${\C{ SL}}$ denote the subvariety of $\C{I}$ defined by
\begin{enumerate}[{\rm(1)}]  
\item   $ x' \approx x,$
\item   $ x \to  y \approx y \to x $ \quad (C).
\end{enumerate}
That is, $\C{SL} :=  \C{I}_{1,0} \cap  \C{C}$.

\begin{lemma} \label{B}
$\C{SL} \models x \to x \approx x$. 
\end{lemma}

\begin{proof}
Let $\mathbf{A}  \in \C{SL}$.   
Then           
\quad $ x \to x   \approx  (x' \to x')'   
\overset{(C)}{\approx}  [(x \to 0) \to (0 \to x)']'  $ 
\begin{align*}
\noindent &\overset{(I)}{\approx} (0 \to 0) \to x  
	\approx  0 \to x  
	\approx  x'  
	\approx  x,
\end{align*} 
proving the lemma.
\end{proof}

Let $\C{S}^0$ denote the variety of $\lor$-semilattices with the least element $0$; that is,  $\mathbf{A} = \langle A, \lor, 0 \rangle \in \C{S}^0$ satisfies:
\begin{enumerate}[{\rm(1)}]   
	\item   $ (x \lor y) \lor z \approx  x \lor (y \lor z) $,
	\item   $ x \lor x \approx  x $,
	\item   $  x \lor y \approx  y \lor x$,
	\item   $  x \lor 0 \approx  x$.
\end{enumerate}
The following lemma is easy.

\begin{lemma}\label{C}
	Let $\mathbf{A} \in \C{S}^0$.  Then $\mathbf{A}$ satisfies:\\
	$[(z' \lor x) \lor (y \lor z)']' \approx  (x \lor y) \lor z $, where $x' := x \lor 0$.
\end{lemma}

\begin{theorem}
	${\C{SL}}$ is term-equivalent to (in fact, is) $\C{S}^0$.  More precisely,
	\begin{enumerate}[{\rm(a)}]   
		\item
		Let $\mathbf{A} \in {\C{SL}}$.  Define $\mathbf{A^{\lor}} := \langle A, \lor, 0 \rangle$ be the algebra where $\lor := \to$.
		Then $\mathbf{A^{\lor}} \in \C{S}^0$. 
		\item
		Let $\mathbf{A} \in \C{S}^0$.  Define $\mathbf{A^{\to}} := \langle A, \to, 0 \rangle$ be the algebra where $\to := \lor$.
		Then $\mathbf{A^{\to}} \in {\C{SL}}$.
		\item
		Let $\mathbf{A} \in {\C{SL}}$.  Then $\mathbf{A}{^{\lor}}^{\to} = \mathbf{A}$.
		\item
		Let $\mathbf{A} \in \C{S}^0$.  Then $\mathbf{A}{^{\to}}^{\lor} = \mathbf{A}$.
	\end{enumerate}
\end{theorem}

\begin{proof}
	(a) follows from Lemma \ref{B} and Theorem \ref{Theo_I10_IDcapAssoc},  
and (b) follows from Lemma \ref{C}.  The rest of the proof is left to the reader.
\end{proof}

The following corollary gives an axiomatization for $\C{V}(\mathbf{2_s})$.
\begin{corollary}\label{CorSL}
	$\C{V}(\mathbf{2_s}) = \C{S}^0 = \C{C} \cap \C{I}_{1,0} = {\C{SL}}$.
\end{corollary}

\begin{corollary}
	The class of congruence lattices $\mathbf{Con\ A}$, where $\mathbf{A} \in \C{I}$, does not satisfy any non-trivial lattice identities.  In particular, the variety $\C{I}$ is neither congruence-distributive, nor congruence-modular.
\end{corollary}

\begin{proof}
It is well known (see \cite{FrNa73}) that the class of congruence lattices of semilattices does not satisfy any nontrivial lattice identities.
\end{proof}

\subsection{A characterization of the variety  $\C{V}(\mathbf{2_z})$}

Recall from \cite{sankappanavarMorgan2012} $\C{Z}$ denotes the subvariety of $\C{I}$ defined by the identity:
\[ x \to y \approx 0.\]
Let $\mathbf{Eq(A)}$ denote the lattice of equivalence relations of $\mathbf{A}$.\\

The following lemma is obvious.
\begin{lemma}
	Let $\mathbf{A} \in \C{Z}$.  Then
	\begin{enumerate}[{\rm(1)}]   
		\item
		$\mathbf{Con\ A} =  \mathbf{Eq(A)}$,
		\item  $\mathbf{2_z}$ is the only nontrivial subdirectly irreducible algebra in $\C{Z}$.
	\end{enumerate}
\end{lemma}

\begin{corollary}
	$\C{Z} = {\C V(\mathbf{2_z)}}$.
\end{corollary}

We end this section by noting that the variety $\C V(\mathbf{2_b})$, generated by the 2-element Boolean algebra $\mathbf{2_b}$, was axiomatized in \cite{sankappanavarMorgan2012}.

\section{The varieties $\C{MC}$,  $\C{MID}$, ${\C A}$ and $\C{C}$}\label{S:12}

Our goal, in this section, is to prove that 
\[  {\C{MC}}  \cap   \C{MID}  \cap  {\C A}  =  {\C SL} 
 \subset   \C{CP}. \]

\begin{lemma} \label{Lemma_on_MC_MID_Assoc}
	Let $\mathbf A \in  {\C{MC}}  \cap   \C{MID}  \cap  {\C A}$.     
		Then $\mathbf A$ satisfies
	\begin{enumerate}[{\rm(a)}]  
		\item   $x \to (x \to 0') \approx x$,  
		\item   $x \lor y \approx x \to [0 \to \{y \to (0 \to (0 \to 0'))\}]$, 
		\item   $x \lor y \approx x \to (0 \to y'), $ 
		\item   $0 \lor 0 \approx 0$, 
		\item   $0' \approx 0$,  
		\item   $x \to x'  \approx x $, \label{180416_01}
		\item   $x \to [y  \to (x \to y')] \approx x  \to y$,  
		\item   $x \to y' \approx y  \to x'$ {\rm (CP)}, 
		\item   $x \to (y \to x') \approx y \to x$, 
		\item   $x' \approx x$. 	
	\end{enumerate}
\end{lemma}

\begin{proof}
		(a):	
\quad	$	x \to (x \to 0') \overset{(A)}{\approx}  x  \to [(x \to 0) \to 0] $  
\begin{align*}
\noindent&\overset{(A)}{\approx}  (x  \to x') \to 0     
\overset{\text{def of }\land}{\approx}  x  \land x   
\overset{ (MID)}{\approx}  x.  
\end{align*}	
		
		(b):
\quad	$x \lor y  \approx  (x' \to y'')''  
\overset{(A)}{\approx}  [x' \to (y \to 0')]'' $ 
\begin{align*}
&\overset{(A)}{\approx}   [x \to \{0 \to (y  \to 0')\}]''  
\overset{(A)}{\approx}  [x \to \{0 \to (y  \to 0')\}']' \\  
&\overset{(A)}{\approx}   [x \to \{0 \to (y  \to 0')'\}]'  
\overset{(A)}{\approx}  [x \to \{0 \to (y  \to 0'')\}]' \\ 
&\overset{(A)}{\approx}  x \to [0 \to (y  \to 0'')]'   
\overset{(A)}{ \approx}  x \to [0 \to (y  \to 0'')'] \\  
&\overset{(A)}{\approx}  x \to [0 \to (y  \to 0''')]   
\overset{(A)}{ \approx}  x \to [0 \to \{y  \to (0 \to (0 \to 0'))\}] .    
\end{align*}

		(c):  This is immediate from (a) and (b).
		
		(d):  From (c) we have $0 \lor 0 = 0 \to (0 \to 0')$.  Now use (a).
		
		(e): From associativity we get $0 \to 0'=0$.  Hence, $0'=0 \to (0 \to 0') = 0$ by (a).
		
		(f):  This is immediate from (a) and (e).
		
		(g):
	$	x \to [y  \to (x \to y')]   \overset{(A)}{\approx}  (x \to y)  \to (x \to y)'   
\overset{(f)}{ \approx}  x \to y.$

		(h):
\quad	$	x \to y'  \overset{(e)}{\approx}   x \to (y \to 0')  
\overset{(A)}{ \approx}  x \to y'' $
\begin{align*}
\noindent&\overset{(A)}{ \approx}   (x \to y')'    
\approx  x \land y                	
\overset{(MC)}{\approx}  y \land x.    
\end{align*}
		Hence,
		$x \to y'
\overset{(A)}{ \approx}  y  \to x''       
\overset{(A)}{ \approx}   y  \to (x \to 0')   
\overset{(e)}{ \approx}  y  \to x' .$

		(i):
	$ x \to (y \to x')   \overset{(A)}{ \approx}  x \to (y \to x)' $   
\begin{align*}	
&\overset{(h)}{\approx}  (y \to x) \to x'         
		\overset{(A)}{ \approx}  y \to  (x \to x')    
 \overset{(\ref{180416_01})}{\approx}  y \to x.    
\end{align*}

		(j): 
	$	x'  \overset{(g)}{\approx}  x \to [0 \to (x \to 0')]   
\overset{(i)}{\approx}  x \to (x \to 0)    
\overset{(f)}{ \approx}   x.$   			
\end{proof}

\begin{theorem} \label{Theorem_MC_MID_in_Assoc_C} 
	$\C{MC}  \cap   \C{MID}  \cap  {\C A}  \subseteq  \C{C} \cap \C{I}_{1,0}   \cap  \C{CP}$.	
\end{theorem}

\begin{proof}
	Let $\mathbf{A}  \in  \C{MC}  \cap  {\C MID}  \cap  {\C A}$.   
	Then
$	x \to y  \overset{\ref{Lemma_on_MC_MID_Assoc}(j)}{\approx} x \to y'   
\overset{\ref{Lemma_on_MC_MID_Assoc}(h)}{\approx}  y \to x'    
\overset{\ref{Lemma_on_MC_MID_Assoc}(j)}{\approx}  y \to x.$ 
In view of (h) and (j) of Lemma \ref{Lemma_on_MC_MID_Assoc}, the proof is now complete. 
\end{proof}

  The following corollary is immediate from the preceding theorem and Corollary \ref{CorSL} and provides another axiomatization of $\C{V(\mathbf{2_s})}$.
\begin{corollary}\label{Cor_MC_ID_Assoc}
	$\C{MC} \cap   \C{MID}  \cap  {\C A}  =  \C{SL}  =  \C{V(\mathbf{2_s})} \subset  \C{CP}$.
\end{corollary}

\section{The varieties $\C{MC}$, $\C{ID}$, $\C{C}$ and $\C{I}_{1,0} $}\label{S:13}

The purpose of this section is to show that $\C{MC}  \cap   \C{ID}   =   \C{C}  \cap 
 \C{I}_{1,0} $.

\begin{lemma} \label{Lemma_010715_01}
	Let $\mathbf{A} \in \C{MC}  \cap   \C{ID}$.   Then $\mathbf{A}$ satisfies:
	\begin{enumerate}[{\rm(a)}]  
		\item   $[(0 \to x) \to y'']' \approx (x \to y)'$,   
		\item   $(0 \to x')' \approx x'' $,   
		\item   $(x \to x'')' \approx x'' $,  
		\item   $x'' \approx x'$,
		\item   $[(x \to y) \to z]'  \approx  (x \to y) \to z$.
		\end{enumerate}	
\end{lemma}

\begin{proof} 
	(a) and (b) are immediate from  (I) and the identity (ID).   
	Use both hypotheses and Lemma \ref{SankaLemma7.5}  
	 to prove (c).

	(d):	$x''  \overset{(c) \text{ and } (I)}{\approx}  [(0' \to x) \to x'''']'    
 \overset{\ref{SankaLemma7.5}}{\approx}  [(0 \to x) \to x'']'  
 \overset{(a)}{\approx}  [x \to x]'   
\overset{(ID)}{\approx}  x'.$

		 (e):
	$	[(x \to y) \to z]'  \overset{(I)}{\approx} [(z' \to x) \to (y \to z)']'' $

$\overset{(d)}{\approx}   [(z' \to x) \to (y \to z)']'       
\overset{(I)}{\approx}  (x \to y) \to z. $
This completes the proof.
\end{proof}

\begin{theorem} \label{Theo_MC_ID_subset_I10_C}
		$\C{MC}  \cap   \C{ID}  =   \C{I}_{1,0}  \cap  \C{C}$. 
\end{theorem}

\begin{proof}
First, we wish to prove $\C{MC}  \cap   \C{ID} \subseteq  \C{I}_{1,0}  \cap  \C{C}$.  
By (ID) we have 
$	x' \approx  (x \to x)'  
	 \approx  [(x \to x) \to x]' $  
Hence, $x' \approx   [(x \to x) \to x]   
	 \approx  x$  by Lemma \ref{Lemma_010715_01} (e)  and (ID).

Next, from $x \to y \approx (x \to y)' \approx (x \to y')' \approx x \land y \approx y \land x \approx (y \to x')' \approx y \to x' = y \to x$, we conclude $\mathbf{A} \in \C{C}$.  Thus we have, in view of Corollary \ref{CorSL}, that  
$\C{MC}  \cap   \C{ID}   \subseteq   \C{I}_{1,0}  \cap  \C{C} = 
\C{V(\mathbf{2_s})}  \subseteq  \C{MC} \cap   \C{ID}$, completing the proof.
\end{proof}

The following corollary is immediate from the preceding theorem and Corollary 
\ref{Cor_MC_ID_Assoc}.

\begin{corollary}

$\C{MC}  \cap   \C{ID} = \C{V(\mathbf{2_s})}  =   \C{MC}  \cap   \C{MID}  \cap  {\C A}$.  
\end{corollary}

\section{Concluding remarks}

These investigations are continued in \cite{CoSa2015b}, \cite{CoSa2015c} and \cite{CoSa2015a}.   
(Note that the implicator groupoids are referred to in those papers as ``implication zroupoids.)  In \cite{CoSa2015b} it is proved  
that the variety $\C{I}_{2,0}$ is a maximal
subvariety of $\C{I}$ with respect to the property that the relation $\leq$, which is defined as follows: 
\[
\text{$x \leq y$ if and only if $(x \to y')' =x$, for $x,y \in \mathbf{A}$ and $\mathbf{A} \in \C{I}$,}
\]
is a partial order. Furthermore, all the finite $\C{I}_{2,0}$-chains, relative to this order, are determined.  
In \cite{CoSa2015c} we describe  all simple algebras in $\C{I}$ and, consequently, we give a description of 
semisimple subvarieties of $\C{I}$.

\cite{CoSa2015a} is a further addtion to the series \cite{sankappanavarMorgan2012},  \cite{CoSa2015b}, \cite{CoSa2015c} and the present paper.  
It studies the structure of the derived algebras $\mathbf{A^{m}} := \langle A, \land, 0 \rangle$ 
and $\mathbf{A^{mj}} :=\langle A, \land, \lor, 0 \rangle$ of $\mathbf{A} \in \mathcal{I}$. 
It also introduces new subvarieties of $\C{I}$ and determines   
their relationships with other subvarieties, both old and new.  In fact, it is shown that, for each $\mathcal{I}$-zroupoid $\mathbf A$, $\mathbf{A}^m$ is a semigroup, which,   
together with Theorem \ref{T:7.3} of this paper  
implies that, for $\mathbf{A} \in \mathcal{I}_{2,0} \cap \mathcal{MC}$, the derived algebra  
$\mathbf{A^{mj}}$ is both a distributive bisemilattice and a Birkhoff system.  
It is also shown in \cite{CoSa2015a} that $\mathcal{CLD} \subset \mathcal{SRD} \subset \mathcal{RD}$, where the varieties $\mathcal{CLD}$, $\mathcal{SRD}$ and $\mathcal{RD}$ are respectively 
defined, relative to $\mathcal{I}$, by: (CLD)  $x \to (y \to z) \approx (x \to z) \to (y \to x)$,  (SRD)    $(x \to y) \to z \approx (z \to x) \to (y \to z)$,  and (RD)  $(x \to y) \to z \approx (x \to z) \to (y \to z)$.  Furthermore, \cite{CoSa2015a} shows the following relationship among some of of the varieties investigated in this paper and $\mathcal{CLD}$:  $\mathcal{C} \subset \ \mathcal{CP} \cap \mathcal{A} \cap \mathcal{MC} \cap \mathcal{CLD}$.  
Both of the results just mentioned are much stronger than the ones that were announced in \cite{sankappanavarMorgan2012}.

\subsection*{Acknowledgment} 
The first author wants to thank the
institutional support of CONICET  (Consejo Nacional de Investigaciones Cient\'ificas y T\'ecnicas).   The authors wish to express their indebtedness to the anonymous referees for their careful reading of the paper.  The authors also wish to acknowledge that \cite{Mc} was a useful tool during the research phase  of this paper.


\section{Appendix}

\begin{proof} {\bf of Lemma \ref{general_properties2}}
	
	Let $a,b,c,d \in A$.
	\begin{itemize}
		\item[(\ref{281014_05})]
		$$
		\begin{array}{lcll}
		(a \to 0') \to b  & = & [(b' \to a) \to (0' \to b)']' & \mbox{from (I)} \\
		& = & [(b' \to a) \to b']' & \mbox{by Lemma \ref{general_properties_equiv} (\ref{TXX})} \\
		& = & [(b' \to a) \to (b' \to b)']' & \mbox{by Lemma \ref{general_properties_equiv} (\ref{LeftImplicationwithtilde})} \\
		& = & (a \to b') \to b. & \mbox{}
		\end{array}
		$$
		\item[(\ref{291014_02})]
		$$
		\begin{array}{lcll}
		a \to (0 \to a)'  & = & (a' \to a) \to (0 \to a)' & \mbox{by Lemma \ref{general_properties_equiv} (\ref{LeftImplicationwithtilde})} \\
		& = & [(a' \to a) \to (0 \to a)']'' & \mbox{} \\
		& = & [(a \to 0) \to a]' & \mbox{using (I)} \\
		& = & [a' \to a]' & \mbox{} \\
		& = & a' & \mbox{by Lemma \ref{general_properties_equiv} (\ref{LeftImplicationwithtilde})}.
		\end{array}
		$$

		\item[(\ref{291014_10})]  
		$$
		\begin{array}{lcll}
		[(b \to a) \to b]' & = & [(b' \to b) \to (a \to b)']'' & \mbox{by (I)} \\
		& = & (b' \to b) \to (a \to b)' & \mbox{} \\
		& = & b \to (a \to b)' & \mbox{by Lemma \ref{general_properties_equiv} (\ref{LeftImplicationwithtilde})} \\
		& = & b'' \to (a \to b)' & \mbox{} \\
		& = & (b' \to 0)' \to (a \to b)'  & \mbox{} \\
		& = & [(b' \to 0)' \to (a \to b)']'' & \mbox{} \\
		& = & [(0 \to a) \to b]' & \mbox{by (I)}.
		\end{array}
		$$
		Hence $(b \to a) \to b = [(b \to a) \to b]'' = [(0 \to a) \to b]'' = (0 \to a) \to b$.\\
		
		\item[(\ref{071114_02})]  
		$$
		\begin{array}{lcll}
		[(a \to 0') \to b]'  & = & [(b' \to a) \to (0' \to b)']'' & \mbox{from (I)} \\
		& = & [(b' \to a) \to b']'' & \mbox{by Lemma \ref{general_properties_equiv} (\ref{TXX})} \\
		& = & (b' \to a) \to b' & \mbox{} \\
		& = & (0 \to a) \to b' & \mbox{by (\ref{291014_10})}.
		\end{array}
		$$
		
		\item[(\ref{311014_03})]
		$$
		\begin{array}{lcll}
		0 \to a   & = & a' \to 0' & \mbox{by Lemma \ref{general_properties} (\ref{cuasiConmutativeOfImplic2})} \\
		& = & (a' \to 0')'' & \mbox{} \\
		& = & [0' \to (a' \to 0')']' & \mbox{by Lemma \ref{general_properties_equiv} (\ref{TXX})} \\
		& = & [(0 \to 0) \to (a' \to 0')']' & \mbox{} \\
		& = & (0 \to a') \to 0' & \mbox{by (I)} \\
		& = & (a \to 0') \to 0' & \mbox{by Lemma \ref{general_properties} (\ref{cuasiConmutativeOfImplic2})} \\
		& = & (a'' \to 0') \to 0' & \mbox{} \\
		& = & [(a' \to 0) \to 0'] \to 0' & \mbox{} \\
		& = & [(a' \to 0'') \to 0'] \to 0' & \mbox{} \\
		& = & [(a' \to 0') \to 0'] \to 0' & \mbox{(\ref{281014_05}) with } y=0'  \\
		& = & [(a' \to 0') \to 0''] \to 0' & \mbox{(\ref{281014_05}) with } y=0' \\
		& = & [(a' \to 0') \to 0] \to 0' & \mbox{} \\
		& = & (a' \to 0')' \to 0' & \mbox{} \\
		& = & 0 \to (a' \to 0')'' & \mbox{by Lemma \ref{general_properties} (\ref{cuasiConmutativeOfImplic2})} \\
		& = & 0 \to (a' \to 0') & \mbox{} \\
		& = & 0 \to (0 \to a) & \mbox{by Lemma \ref{general_properties} (\ref{cuasiConmutativeOfImplic2})}.
		\end{array}
		$$
		
		\item[(\ref{031114_06})]  
		
		$$
		\begin{array}{lcll}
		[a' \to (0 \to b)]'  & = & [(a \to 0) \to (0 \to b)]' & \mbox{} \\
		& = & [\{(0 \to b)' \to a\} \to \{0 \to (0 \to b)\}']'' & \mbox{from (I)} \\
		& = & [(0 \to b)' \to a] \to [0 \to (0 \to b)]' & \mbox{} \\
		& = & [(0 \to b)' \to a] \to (0 \to b)' & \mbox{by (\ref{311014_03})} \\
		& = & [0 \to a] \to (0 \to b)' & \mbox{by (\ref{291014_10})}.
		\end{array}
		$$

		\item[(\ref{291014_09})] By Lemma \ref{general_properties_equiv} (\ref{LeftImplicationwithtilde}) and (I) we have that $[a \to (b \to a)']' = [(a' \to a) \to (b \to a)']' = (a \to b) \to a$.  
		
		\item[(\ref{031114_05})]
		$$
		\begin{array}{lcll}
		0 \to [(0 \to a) \to (0 \to b')'] & = & 0 \to [(0 \to a) \to (b \to 0')'] & \mbox{by Lemma \ref{general_properties} (\ref{cuasiConmutativeOfImplic2})} \\
		& = & 0 \to [(0 \to a) \to (b \to 0')']'' & \mbox{} \\
		& = & 0 \to [(a \to b) \to 0']' & \mbox{from (I)} \\
		& = & [(a \to b) \to 0'] \to 0'  & \mbox{by Lemma \ref{general_properties} (\ref{cuasiConmutativeOfImplic2})} \\
		& = & [(a \to b) \to 0''] \to 0' & \mbox{by  (\ref{281014_05})} \\
		& = & [(a \to b) \to 0] \to 0' & \mbox{} \\
		& = & (a \to b)' \to 0' & \mbox{} \\
		& = & 0 \to (a \to b) & \mbox{by Lemma \ref{general_properties} (\ref{cuasiConmutativeOfImplic2})}.
		\end{array}
		$$
		
		\item[(\ref{071114_03})]
		$$
		\begin{array}{lcll}
		0 \to ((0 \to a) \to b') & = & 0 \to [(a' \to 0') \to b'] & \mbox{by Lemma \ref{general_properties} (\ref{cuasiConmutativeOfImplic2})} \\
		& = & 0 \to [(a' \to 0') \to b']'' & \mbox{} \\
		& = & 0 \to [(0 \to a') \to b'']' & \mbox{by (\ref{071114_02})} \\
		& = &  0 \to [(0 \to a') \to b]' & \mbox{} \\
		& = & [(0 \to a') \to b] \to 0' & \mbox{by Lemma \ref{general_properties} (\ref{cuasiConmutativeOfImplic2})} \\
		& = & [(0 \to (0 \to a')) \to (b \to 0')']' & \mbox{from (I)} \\
		& = & [(0 \to a') \to (b \to 0')']' & \mbox{by (\ref{311014_03})} \\
		& = & (a' \to b) \to 0' & \mbox{from (I)} \\
		& = & 0 \to (a' \to b)' & \mbox{by Lemma \ref{general_properties} (\ref{cuasiConmutativeOfImplic2})} \\
		& = & a \to (0 \to b'). & \mbox{}
		\end{array}
		$$
		
		\item[(\ref{031114_07})]
		$$
		\begin{array}{lcll}
		0 \to (0 \to a)' & = & (0 \to a) \to 0' & \mbox{by Lemma \ref{general_properties} (\ref{cuasiConmutativeOfImplic2})} \\
		& = & (0' \to a) \to 0' & \mbox{by (\ref{291014_10})} \\
		& = & a \to 0' & \mbox{} \\
		& = & 0 \to a' & \mbox{by Lemma \ref{general_properties} (\ref{cuasiConmutativeOfImplic2})}.
		\end{array}
		$$
		
		\item[(\ref{071114_04})]  
		$$
		\begin{array}{lcll}
		0 \to (a \to b) & = & 0 \to [(0 \to a) \to (0 \to b')'] & \mbox{by (\ref{031114_05})} \\
		& = & a \to [0 \to (0 \to b')'] & \mbox{by (\ref{071114_03}) with } x = a, y = 0 \to b' \\
		& = & a \to (0 \to b'') & \mbox{by (\ref{031114_07})} \\
		& = & a \to (0 \to b). & \mbox{}
		\end{array}
		$$
		
		\item[(\ref{291014_03})]
		$$
		\begin{array}{lcll}
		[(b \to a) \to (0 \to a)']' & = & [\{(0 \to a)'' \to b\} \to \{a \to (0 \to a)'\}']'' & \mbox{from (I)} \\
		& = & [(0 \to a)'' \to b] \to [a \to (0 \to a)']' & \mbox{} \\
		& = & [(0 \to a) \to b] \to [a \to (0 \to a)']' & \mbox{} \\
		& = & [(0 \to a) \to b] \to a'' & \mbox{by (\ref{291014_02})} \\
		& = & [(0 \to a) \to b] \to a. & \mbox{}
		\end{array}
		$$
		
		\item[(\ref{291014_05})] It follows immediately from (\ref{291014_03}) since $[[(0 \to a) \to b] \to a]' = [(b \to a) \to (0 \to a)']'' = (b \to a) \to (0 \to a)'$.\\
\ \\ \ \\	 \ \\	
		\item[(\ref{281014_07})]	
	        
		\begin{align*}  
		(0 \to a') \to (b \to a)  &=  (a \to 0') \to (b \to a)\\ 
		&  \qquad \qquad \qquad \text{by Lemma \ref{general_properties} (\ref{cuasiConmutativeOfImplic})} \\
		&=  (a'' \to 0') \to (b \to a)   & \mbox{} \\
		& =  [(a' \to 0) \to 0'] \to (b \to a)   & \mbox{} \\
		 &=  [(a' \to 0'') \to 0'] \to (b \to a)  & \mbox{} \\
		 &=  [(a' \to 0') \to 0'] \to (b \to a)  & \\
		  & \qquad \qquad \qquad \mbox{by item (\ref{281014_05}) with } x=a' \mbox{ and } y = 0' \\
		&=  [(a' \to 0') \to (b \to a)'] \to (b \to a) \\
		& \qquad \qquad \qquad \mbox{by item (\ref{281014_05}) with }  x=a' \to 0', y = b \to a \\
		 &=  [(a' \to 0') \to (b \to a)'] \to [(0' \to b) \to a]  \\
		 & \qquad \qquad \qquad \mbox{by Lemma \ref{general_properties_equiv} (\ref{TXX})} \\
		 &=  (0' \to b) \to a  \\ 
		  & \qquad \qquad \qquad \mbox{by item (\ref{281014_05}) with }  x=a, y = 0', z=b \\
		&=  b \to a   
		\end{align*}

		\item[(\ref{291014_04})]
		$$
		\begin{array}{lcll}
		a' \to (0 \to a) & = & [a \to (0 \to a)'] \to (0 \to a) & \mbox{by (\ref{291014_02})} \\
		& = & [a \to 0'] \to (0 \to a) & \mbox{from (\ref{281014_05}) using } x = a, y = 0 \to a  \\
		& = & [0 \to a'] \to (0 \to a) & \mbox{from Lemma \ref{general_properties} (\ref{cuasiConmutativeOfImplic})} \\
		& = & 0 \to a & \mbox{from (\ref{281014_07}) using } x = a, y = 0.
		\end{array}
		$$
		
		\item[(\ref{291014_06})]
		$$
		\begin{array}{lcll}
		(b \to a)' & = & [(0' \to b) \to a]' & \mbox{using Lemma \ref{general_properties_equiv} (\ref{TXX})} \\
		& = & [(a' \to 0') \to (b \to a)']'' & \mbox{by (I)} \\
		& = & (a' \to 0') \to (b \to a)' & \mbox{} \\
		& = & (0 \to a) \to (b \to a)' & \mbox{by Lemma \ref{general_properties} (\ref{cuasiConmutativeOfImplic2})}.
		\end{array}
		$$
		
		\item[(\ref{291014_07})]  
		\begin{align*}  
		(a \to b) \to (0 \to b)' & =  [\{(0 \to b) \to a\} \to b]'   \qquad \qquad  \mbox{by (\ref{291014_05}) using } x = b, y = a \\
		& =  [\{b' \to (0 \to b)\} \to (a \to b)']''  \qquad \mbox{by (I)} \\
		& =  [b' \to (0 \to b)] \to (a \to b)' & \mbox{} \\
		& =  (0 \to b) \to (a \to b)'   \qquad \qquad \qquad \mbox{by (\ref{291014_04}) using } x = b \\
		& =  (a \to b)'  \qquad \qquad \qquad \mbox{by (\ref{291014_06}) using }  x=b, y = a.
		\end{align*}

		\item[(\ref{071114_01})]  
		$$
		\begin{array}{lcll}
		a \to (0 \to b') & = & [a \to (0 \to b')]'' & \mbox{} \\
		& = & [a'' \to (0 \to b')]'' & \mbox{} \\
		& = & [(0 \to a') \to (0 \to b')]' & \mbox{by (\ref{031114_06})} \\
		& = & [(0 \to a') \to (b \to 0')]' & \mbox{by Lemma \ref{general_properties} (\ref{cuasiConmutativeOfImplic2})} \\
		& = & (a' \to b) \to 0' & \mbox{from (I)} \\
		& = & 0 \to (a' \to b)' & \mbox{by Lemma \ref{general_properties} (\ref{cuasiConmutativeOfImplic2})}.
		\end{array}
		$$

                 \item[(\ref{271114_03})]
		
		\begin{align*}  
		[(a \to b) \to a] \to [(b \to a) \to b] & =  [(a \to b) \to a] \to [b \to (a \to b)']'  \qquad \mbox{by (\ref{291014_09})} \\
		& =  [(a \to b) \to (a \to b)']'  \qquad \mbox{by (I)} \\
		& =  (a \to b)''  \qquad \mbox{by Lemma \ref{general_properties_equiv} (\ref{LeftImplicationwithtilde})} \\
		& =  a \to b. & \mbox{}
		\end{align*}

		\item[(\ref{271114_04})]
		$$
		\begin{array}{lcll}
		[a \to (b \to a')] \to a & = & [0 \to (b \to a')] \to a & \mbox{by (\ref{291014_10})} \\
		& = & [b \to (0 \to a')] \to a & \mbox{by (\ref{071114_04})} \\
		& = & [(a' \to b) \to \{(0 \to a') \to a\}']' & \mbox{by (I)} \\
		& = & [(a' \to b) \to \{(a \to a') \to a\}']' & \mbox{by (\ref{071114_04})} \\
		& = & [(a' \to b) \to (a' \to a)']' & \mbox{by Lemma \ref{general_properties_equiv} (\ref{LeftImplicationwithtilde})} \\
		& = & [(a' \to b) \to a']' & \mbox{by Lemma \ref{general_properties_equiv} (\ref{LeftImplicationwithtilde})} \\
		& = & [a' \to (b \to a')']'' & \mbox{by (\ref{291014_09})} \\
		& = & a' \to (b \to a')'. & \mbox{}
		\end{array}
		$$

\item[(\ref{281114_01})]
		
		\begin{align*}      
		a \to (b \to a') & =  [(a \to (b \to a')) \to a] \to [\{(b \to a') \to a\} \to (b \to a')] \qquad \mbox{by (\ref{271114_03})} \\
		& =  [\{a \to (b \to a')\} \to a] \to [(0 \to a) \to (b \to a')] \qquad \qquad \mbox{by (\ref{291014_10})} \\
		& =  [\{a \to (b \to a')\} \to a] \to (b \to a') \qquad \qquad \mbox{by (\ref{281014_07})} \\
		& =  [a' \to (b \to a')'] \to (b \to a') \qquad \qquad  \mbox{by (\ref{271114_04})} \\
		& =  (a' \to 0') \to (b \to a') \qquad \qquad \mbox{by (\ref{281014_05})} \\
		& =  (0 \to a) \to (b \to a') \qquad \qquad \mbox{by Lemma \ref{general_properties} (\ref{cuasiConmutativeOfImplic2})} \\
		& =  b \to a' \qquad \qquad \qquad \qquad \mbox{by (\ref{281014_07})}.
		\end{align*}

		\item[(\ref{191114_05})]
		$$
		\begin{array}{lcll}
		0 \to (a \to b')' & = & 0 \to (a'' \to b')' & \mbox{} \\
		& = & a' \to (0 \to b'') & \mbox{by (\ref{071114_01})} \\
		& = & a' \to (0 \to b) & \mbox{} \\
		& = & 0 \to (a' \to b) & \mbox{by (\ref{071114_04})}.
		\end{array}
		$$
		
		\item[(\ref{071114_05})] 
		$$
		\begin{array}{lcll}
		(a \to b) \to b' & = & [(a \to b) \to b']'' & \mbox{} \\
		& = & [(b \to a) \to b]' & \mbox{by (\ref{291014_09})} \\
		& = & [(b \to a) \to b'']' & \mbox{} \\
		& = & [(b \to a) \to (b'' \to b')']' & \mbox{by Lemma \ref{general_properties_equiv} (\ref{LeftImplicationwithtilde})} \\
		& = & [(b \to a) \to (b \to b')']' & \mbox{} \\
		& = & (a \to b) \to b' & \mbox{by (I)}.
		\end{array}
		$$
		
		\item[(\ref{181114_04})]
		$$
		\begin{array}{lcll}
		a \to [(b \to c') \to a]' & = & [\{a \to (b \to c')\} \to a]' & \mbox{by (\ref{291014_09})} \\
		& = & [\{0 \to (b \to c')\} \to a]' & \mbox{by (\ref{291014_10})} \\
		& = & [\{b \to (0 \to c')\} \to a]' & \mbox{by (\ref{071114_04})} \\
		& = & (a' \to b) \to [(0 \to c') \to a]' & \mbox{by (I)} \\
		& = & (a' \to b) \to [(c \to 0') \to a]' & \mbox{by Lemma \ref{general_properties} (\ref{cuasiConmutativeOfImplic2})} \\
		& = & (a' \to b) \to [(0 \to c) \to a'] & \mbox{by (\ref{071114_02})}.
		\end{array}
		$$
		
		\item[(\ref{181114_05})]
		First we have that
		$$
		\begin{array}{lcll}
		\{(0 \to a) \to b'\}' \to c & = & [(a' \to 0') \to b']' \to c  & \mbox{by Lemma \ref{general_properties} (\ref{cuasiConmutativeOfImplic2})} \\
		& = & [(0 \to a') \to b] \to c & \mbox{by (\ref{071114_02})} \\
		& = & [(a \to 0') \to b] \to c. & \mbox{by Lemma \ref{general_properties} (\ref{cuasiConmutativeOfImplic2})}.
		\end{array}
		$$
		Hence
		\begin{equation} \label{181114_06}
			[(0 \to a) \to b']' \to c  =  [(a \to 0') \to b] \to c.
		\end{equation}
                 Then
		\begin{align*}  
		(c \to d) \to [(0 \to a) \to b'] \\
		& =  [\{((0 \to a) \to b')' \to c\} \to \{d \to ((0 \to a) \to b')\}']'  \\
		&  \qquad \qquad \qquad  \mbox{by (I)} \\
		& =  [\{((a \to 0') \to b) \to c\} \to \{d \to ((0 \to a) \to b')\}']' \\
		&  \qquad \qquad  \mbox{by equation (\ref{181114_06})}.
		\end{align*}

                  \item[(\ref{080415_01})]
				
				\begin{align*} 
				(c \to a) \to (b \to c)	& =  [\{(b \to c)' \to c\} \to \{a \to (b \to c)\}']' & \mbox{} \\
				& =  [\{(c' \to (b \to c)) \to (0 \to c)'\}' \to \{a \to (b \to c)\}']'  \quad \quad \mbox{by (I)} \\
			& =  [\{(b \to c) \to (0 \to c)'\}' \to \{a \to (b \to c)\}']'  \qquad \quad  \mbox{by  (\ref{281114_01})} \\
				& =  [(b \to c)'' \to \{a \to (b \to c)\}']'  \qquad \qquad \mbox{by (\ref{291014_07})} \\
				& =  [\{(b \to c)' \to 0\} \to \{a \to (b \to c)\}']' & \mbox{} \\
				& =  (0 \to a) \to (b \to c) \qquad \qquad \mbox{by (I)}.
				\end{align*}

		\item[(\ref{181114_07})]
		Observe that
		\begin{equation} \label{181114_09}
			[(c \to 0') \to a] \to a' = (0 \to c) \to a'
		\end{equation}
		since
		$$
		\begin{array}{lcll}
		[(c \to 0') \to a] \to a'  & = & [(c \to 0') \to 0'] \to a' & \mbox{by (\ref{281014_05})} \\
		& = & [(c \to 0) \to 0'] \to a' & \mbox{by (\ref{281014_05})} \\
		& = &  [c' \to 0'] \to a' & \mbox{by definition of } ' \\
		& = & (0 \to c) \to a'. & \mbox{by Lemma \ref{general_properties} (\ref{cuasiConmutativeOfImplic2})}.
		\end{array}
		$$
		Hence
                \begin{align*}  
		(a' \to b) \to [(0 \to c) \to a'] & =  [\{((c \to 0') \to a) \to a'\} \to \{b \to ((0 \to c) \to a')\}']'  \\
		& \qquad \qquad  \qquad \mbox{by (\ref{181114_05})}  
		    \mbox{  with } x = c, y = a, z = a', u = b & \\
		& =  [\{(0 \to c) \to a'\} \to \{b \to ((0 \to c) \to a')\}']' \\
		& \qquad \qquad \qquad \qquad \mbox{by (\ref{181114_09})} \\
		& =  [\{(0 \to c) \to a'\} \to b] \to \{(0 \to c) \to a'\} \quad \mbox{by (\ref{291014_09})} \\
		& =  (0 \to b) \to [(0 \to c) \to a']  \quad \mbox{by (\ref{291014_10})}. 
		\end{align*}

		\item[(\ref{181114_10})]
		
		This is immediate from (\ref{181114_04}) and (\ref{181114_07}).\\

\ \\

                  \item[(\ref{080415_02})]

		\begin{align*}    
		(a' \to b) \to (a \to b')	& =  [b' \to (a' \to b)] \to (a \to b')  \qquad \qquad  \qquad \mbox{by (\ref{281114_01})} \\
		& =  [0 \to (a' \to b)] \to (a \to b')  \qquad \qquad  \qquad \mbox{by (\ref{080415_01})} \\
		& =  [a' \to (0 \to b)] \to (a \to b')   \qquad \qquad  \qquad \mbox{by (\ref{071114_04})} \\
		& =  [a' \to (0 \to b'')] \to (a \to b') & \mbox{} \\
		& =  [0 \to (a'' \to b')'] \to (a \to b')  \qquad \qquad  \qquad \mbox{by (\ref{071114_01})} \\
		& =  [0 \to (a \to b')'] \to (a \to b') & \mbox{} \\
		& =  [(a \to b') \to (a \to b')'] \to (a \to b')  \qquad \qquad  \qquad \mbox{by  (\ref{291014_10})} \\
		& =  [(a \to b')'' \to (a \to b')'] \to (a \to b') & \mbox{} \\
		& =  (a \to b')' \to (a \to b') 
		 \qquad \qquad  \qquad \mbox{by Lemma \ref{general_properties_equiv} (\ref{LeftImplicationwithtilde})} \\
		& =  a \to b' 
		 \qquad \qquad  \qquad   \mbox{by Lemma \ref{general_properties_equiv} (\ref{LeftImplicationwithtilde})}.
		\end{align*}

                 \item[(\ref{080415_03})]
		\begin{align*}    
		(a \to b')' \to (a' \to b)' & =  [\{(a' \to b)'' \to (a \to b')\} \to \{0 \to (a' \to b)'\}']' \\
		& \qquad \mbox{by (I)} \\
		& =  [\{(a' \to b) \to (a \to b')\} \to \{0 \to (a' \to b)'\}']'  \mbox{} \\
		& =  [\{(a' \to b) \to (a \to b')\} \to \{a \to (0 \to b')\}']' \qquad   \mbox{by (\ref{071114_01})} \\
		& =  [\{(a' \to b) \to (a \to b')\} \to \{0 \to (a \to b')\}']'  \qquad \mbox{by  (\ref{071114_04})} \\
		& =  [(a' \to b) \to (a \to b')]''  \qquad  \mbox{by (\ref{291014_07})} \\
		& =  (a' \to b) \to (a \to b') & \mbox{} \\
		& =  a \to b'  \qquad \qquad  \mbox{by (\ref{080415_02})}. \qedhere
		\end{align*} 
	\end{itemize} 
\end{proof}

\ \\

\end{document}